\documentclass{amsart}
\usepackage{amsmath,amsfonts,amssymb,latexsym,epic,eepic}
\usepackage[dvips]{graphics,epsfig,color}
\usepackage{geometry}
\usepackage{graphicx}
\usepackage{pstricks}
\usepackage{mdwlist}
\usepackage{dsfont}
\usepackage[latin1]{inputenc}
\usepackage[T1]{fontenc}
\usepackage[latin1]{inputenc}
\usepackage[T1]{fontenc}
\usepackage[english]{babel}
\usepackage{fancybox}
\usepackage{cancel}
\usepackage{bm}
\usepackage{amssymb,tikz}
\usepackage{cite}
\usepackage{amsthm}
\usepackage{enumerate}

\newtheorem{rmq}{Remark}

\newtheorem{df}{Definition}
\newtheorem{thm}{Theorem}
\newtheorem{prop}{Proposition}
\newtheorem{lm}{Lemma}

\def\ba{\begin{array}{l}}
\def\benum{\begin{enumerate}}
\def\bitem{\begin{itemize}}

\def\cstab{C_s}

\def\div{{\rm div}}

\newcommand{\dgammax}{\ \mathrm{d}\gamma(\bfx)}

\def\ea{\end{array}}
\def\ee{\end{equation}}
\def\eenum{\end{enumerate}}

\def\g{\gamma}

\def\grad{{\nabla}}

\def\nnn{{n \in \N}}

\def\P{{\mathcal P}}

\def\refe#1{(\ref{#1})}

\def\vphi{\varphi}

\def\R{\ifmmode{\rm	I\mkern-3.1mu
R\mkern1mu}\else{\rm I\kern-.18em 
R\hskip1pt\	}\fi\relax}	

\def\Z{\ifmmode{\it	Z\mkern-7.8mu
Z\mkern2mu}\else{\it Z\kern-.28em 
Z\hskip1pt\	}\fi\relax}	

\def\Q{\ifmmode{\rm	Q\mkern-10mu
l\mkern4.5mu}\else{\rm Q\kern-.57em
l\hskip3pt\	}\fi\relax}	

\def\N{\ifmmode{\rm	I\mkern-3.1mu
N\mkern0.5mu}\else{\rm I\kern-.16em
N\hskip0.5pt\ }\fi\relax} 

\def\CC{\ifmmode{\rm C\mkern-8.8mu
l\mkern4mu}\else{\rm C\kern-.48em
l\hskip2.6pt\ }\fi\relax}

\newcommand{\sddi}{ \sum_{\substack{\edged \in \widetilde{\E}^{(i)}_{\text{int}},\\ \edged = D_\edge| D_{\edge'}}}}

\newcommand{\edges}{{\mathcal E}}
\newcommand{\edged}{\epsilon}

\newcommand{\edgesint}{{\mathcal E}_{\rm int}}

\newcommand{\edge}{{\sigma}}

\newcommand{\edgesK}{\edges(K)}
\newcommand{\edgesL}{\edges(L)}
\newcommand{\edgesext}{{\mathcal E}_{\mathrm{ext}}}
\newcommand{\edgesinti}{{\mathcal E}_{\mathrm{int}}^{(i)}}
\newcommand{\edgesintj}{{\mathcal E}_{\mathrm{int}}^{(j)}}
\newcommand{\edgesexti}{{\mathcal E}_{\mathrm{ext}}^{(i)}}
\newcommand{\edgesi}{\mathcal {E}^{(i)}}
\newcommand{\edgesj}{\mathcal {E}^{(j)}}
\newcommand{\edgesdi}{{\edgesd^{(i)}}}
\newcommand{\edgesd}{\widetilde {\edges}}
\newcommand{\edgesdinti}{{\edgesd^{(i)}_{{\rm int}}}}
\newcommand{\edgesdexti}{{\edgesd^{(i)}_{{\rm ext}}}}

\newcommand{\mesh}{{\mathcal M}}

\newcommand{\bfe}{{\boldsymbol e}}
\newcommand{\bff}{{\boldsymbol f}}

\newcommand{\bfn}{{\boldsymbol n}}

\newcommand{\bfu}{{\boldsymbol u}}

\newcommand{\uedge}{{u}_\edge}

\newcommand{\bfv}{{\boldsymbol v}}

\newcommand{\bfw}{{\boldsymbol w}}

\newcommand{\bfx}{{\boldsymbol x}}

\newcommand{\bfvphi}{{\boldsymbol \varphi}}

\newcommand{\vr}{\varrho}

\newcommand{\sth}{\sum_{K \in {\mathcal{M}} }}

\newcommand{\stkl}{\sum_{\edge \in {\mathcal{E}}_{\intt},\edge=K|L} }

\newcommand{\bu}{\bm{u}}
\newcommand{\bv}{\bm{v}}

\newcommand{\bn}{\bm{n}}
\newcommand{\bV}{\bm{V}}

\newcommand{\characteristic}{\mathds{1}}
\newcommand{\ie}{\emph{i.e.\/}}

\newcommand{\dx}{\, {\rm d}\bfx}

\newcommand{\mass}{M}
\newcommand{\nti}{{n \tends + \infty}}
\newcommand{\tends}{\rightarrow}

\newcommand{\qbar}{\delta}

\usepackage[colorlinks=true, pdfstartview=FitV, linkcolor=blue, citecolor=red, urlcolor=blue]{hyperref}
\definecolor{bfonce}{rgb}{0.,0.,0.8}	
\definecolor{rougec}{rgb}{1,0.4,0.}
\definecolor{bclair}{rgb}{0.87,0.92,1.}
\definecolor{bclairr}{rgb}{0.76,0.85,1.}
\definecolor{bclairf}{rgb}{0.40,0.65,0.89}
\definecolor{vertf}{rgb}{0,0.55,0.1}
\definecolor{vertm}{rgb}{0,0.6,0.35}
\definecolor{vclair}{rgb}{0.6,0.95,0.75}
\definecolor{or}{rgb}{0.98,0.6,.1}
\definecolor{shadow}{rgb}{.4,.4,.6}
\newcommand{\xC}{{\rm C}} 
\newcommand{\xH}{{\rm H}}
\newcommand{\xHone}{{\rm H}^{1}}
\newcommand{\xN}{\mathbb{N}}
\newcommand{\xR}{\mathbb{R}}
 
\newcommand{\edgeperp}{\tau}

\newcommand{\fluxd}{F_{\edge,\edged}}
\newcommand{\Ds}{{\scalebox{0.6}{$D_\edge$}}}

\newcommand{\partdtg}{\overline \eth}

\newcommand{\ui}{u_i}
\newcommand{\vi}{v_i}

\newcommand{\ei}{^{(i)}}

\newcommand{\uKedge}{u_{K,\edge}}

\newcommand{\bvarphi}{{\boldsymbol \varphi}}

\newcommand{\disc}{{\mathcal D}}

\newcommand{\nablai}{{\boldsymbol \nabla}}
\newcommand{\gradedges}{{\nablai}_{\! \edges}}

\newcommand{\medge}{\vert \edge \vert}

\newcommand{\norm}[2]{\hspace{.2em}|\hspace{-.1em}| #1 |\hspace{-.1em}|_{#2}\hspace{.2em}}
\newcommand{\normd}[2]{\hspace{.2em}|\hspace{-.1em}| #1 |\hspace{-.1em}|^2_{#2}\hspace{.2em}}
\newcommand{\normexp}[3]{\hspace{.2em}|\hspace{-.1em}| #1 |\hspace{-.1em}|^{#3}_{#2}\hspace{.2em}}

\newcommand{\dive}{{\rm div}}
\newcommand{\dived}{{\rm div}_{\hspace{-0.2em}\raisebox{-0.1em}{\scalebox{0.5}{$\mesh$}}}}
\newcommand{\divedn}{{\rm div}_{\hspace{-0.2em}\raisebox{-0.1em}{\scalebox{0.5}{$\mesh_n$}}}}

\newcommand{\curl}{{\rm curl}}
\newcommand{\curld}{{\rm curl}_{\raisebox{-0.1em}{\scalebox{0.5}{$\mesh$}}}}
\newcommand{\gradi}{{\boldsymbol \nabla}}
\newcommand{\gradtg}{{\overline \nabla}}

\newcommand{\llbracket}{[ \! \!\vert }
\newcommand{\rrbracket}{\vert \! \! ]}

\newcommand{\Hmesh}{{\mathbf{H}_\edges}}

\newcommand{\Hmeshzero}{{\mathbf{H}_{\edges,0}}}

\newcommand{\Hmeshnzero}{{\mathbf{H}_{\edges_n,0}}}
\newcommand{\Hmeshi}{H_\edges^{(i)}}

\newcommand{\Hmeshizero}{H_{\edges,0}^{(i)}}

\DeclareMathOperator{\dv}{div}

\DeclareMathOperator{\dt}{dt}
\DeclareMathOperator{\upw}{up}

\DeclareMathOperator{\diam}{diam}

\DeclareMathOperator{\intt}{int}

\DeclareMathOperator{\E}{{\mathcal{E}}}
\DeclareMathOperator{\Rr}{{\mathcal{R}}}

\numberwithin{equation}{section}
\begin{document}

\title[Convergence of the MAC scheme]{Convergence  of the MAC scheme for the  compressible stationary Navier-Stokes equations}

\author{T. Gallou\"et}
\address{Aix Marseille Univ, CNRS, Centrale Marseille, I2M, Marseille, France}
\email{thierry.gallouet@univ-amu.fr}

\author{R. Herbin}
\address{Aix Marseille Univ, CNRS, Centrale Marseille, I2M, Marseille, France}
\email{raphaele.herbin@univ-amu.fr}

\author{J.-C. Latch\'e}
\address{Institut de Radioprotection et de S\^{u}ret\'{e} Nucl\'{e}aire (IRSN), France}
\email{jean-claude.latche@irsn.fr}

\author{D. Maltese}
\address{IMATH, Université du Sud Toulon-Var,
BP 20132 - 83957 La Garde Cedex, France}
\email{david.maltese@univ-tln.fr}


\begin{abstract}
We prove in this paper the convergence of the Marker and Cell (MAC) scheme for the discretization of the steady state compressible and isentropic Navier-Stokes equations on two or three-dimensional Cartesian grids. Existence of a solution to the scheme is proven, followed by estimates on approximate solutions, which yield the convergence of the approximate solutions, up to a subsequence, and in an appropriate sense. We then prove that the limit of the approximate solutions satisfies the mass and momentum balance equations, as well as the equation of state, which is the main difficulty of this study.
\end{abstract}

\maketitle

%{\bf Keywords:} Compressible fluids, Navier-Stokes equations, Cartesian grids, Marker and Cell scheme, Convergence.
%\\
%{\bf AMS classification} 35Q30, 65N12, 76N10, 76N15, 76M20.
%\titleformat\section{}{}{0pt}{\Large\scshape\bfseries\filcenter}
%
%\titleformat\section{}{}{0pt}{\Large\scshape\bfseries\filcenter\thesection{} - }
%
%\tableofcontents

%
% ----------------------------------------------------------------------------------------------------------------------------------------------
%
\section{Introduction}\label{intro}

The aim of this paper is to prove the convergence of the marker-and-cell (MAC) scheme for the discretization of the stationary and isentropic compressible Navier-Stokes system.
These equations are posed on a bounded domain $\Omega$ of $\R^d$, compatible with a MAC grid (see section \ref{3}), $d=2,3$, and read:
\begin{subequations}\label{pbcont_w}
	 \begin{align}
		&  \dv ( \vr \bu) = 0 \ \text{in} \ \Omega,  \label{cont2} \\
		&\dv(\vr \bu \otimes \bu) - \mu \Delta \bu-(\mu+\lambda)\nabla \dv \bu +  \nabla p = \bm{f} \ \text{in} \ \Omega, \label{mov2} \\
                  & p=\vr^\gamma \ \text{in} \ \Omega, \ \vr \ge 0 \ \text{in} \ \Omega, \ \int_\Omega \vr \dx=M, \label{EOS}
	\end{align}
	\label{pbcont}
\end{subequations}
supplemented by the boundary condition
\begin{equation}\label{ci2}
	\bu_{|\partial \Omega} =0.
\end{equation}
In the above equations, the unknown functions are the scalar density and pressure fields, denoted by $\vr(\bfx)\ge 0$ and $p(\bfx)$ respectively, and the vector velocity field $\bu=(u_1,\ldots,u_d)(\bfx)$, where  $\bfx\in \Omega$ denotes the space variable. 
The viscosity coefficients $\mu$ and $\lambda$ are such that (see \cite{feireisl2004dynamics})
\begin{equation}\label{visc}
	\mu > 0, \qquad \lambda+ \frac {2}{d} \mu \ge 0.
\end{equation}
The function $\bm{f} \in L^2(\Omega)^d$ represents the resultant of the exterior forces acting on the fluid while the constant $ M > 0 $ stands for the total mass of the fluid.
In the compressible barotropic Navier-Stokes equations, the pressure is a given function of the density. Here we assume that the fluid is a perfect gas obeying Boyle's law:
\begin{equation}\label{pressure1}
p = a \vr^\gamma \ \text{in} \ \Omega,
\end{equation}
where $a>0$ and where $\gamma >1$ is termed the {\it{adiabatic constant}}. Typical values of $\gamma$  range from a maximum $5/3$ for {\it{monoatomic gases}}, through $7/5$ for {\it{diatomic gases}} incuding air, to lower values close to $1$ for {\it{polyatomic gases}} at high temperature. For the sake of simplicity, the constant $a$ will be taken equal to $1$. Unfortunately, for purely technical reasons, we will be forced to require that $\gamma >3$ if $d=3$ to prove the convergence of the MAC scheme. There is no restriction if $d=2$ in the sense that we can choose $\gamma >1$.

\begin{rmq}[Forcing term involving the density]
Instead of taking a given function $\bm{f}$ in ($\ref{mov2}$), it is possible, in order to take the gravity effects into account, to take $\bm{f} = \vr \bm{g} $ with $ \bm{g} \in L^\infty (\Omega)^d $.
\end{rmq}

The mathematical analysis of numerical schemes for the discretization of the  steady and/or time-dependent compressible Navier-Stokes and/or compressible Stokes equations has been the object of some recent works. 
The convergence of the discrete solutions to the weak solutions of the compressible stationary Stokes problem was shown for a finite volume-- non conforming P1 finite element
\cite{gal-09-conv,eymard2010convergent,fettah2012numerical} and for the wellknown MAC scheme (see \cite{eymard2010convergence}) which was introduced in \cite{harlow1965numerical} and is widely used in computational fluid dynamics.
The unsteady Stokes problem was also discretized using a FV-FE scheme (Finite Volumes and Finite Elements) on a reformulation of the problem, which were  proven to be convergent \cite{karlsen2012convergent}.
The unsteady barotropic Navier-Stokes equations was also recently tackled in \cite{karper2013convergent}, with a FV-FE scheme, albeit only in the case $\gamma>3$ (there is a real difficulty in the realistic case $\gamma \le 3$ arising from the treatment of the non linear convection term). Some error estimates have been derived for this FV-FE scheme in \cite{Gallouet:2015ab}.

Since the very beginning of the introduction of the Marker-and-Cell (MAC) scheme \cite{harlow1965numerical}, it is claimed that this discretization is suitable for both incompressible and compressible flow problems (see \cite{har-68-num,har-71-num} for the seminal papers, \cite{cas-84-pre,iss-85-sol,iss-86-com,kar-89-pre,bij-98-uni,col-99-pro,van-01-sta,van-03-con,vid-06-sup,wal-02-sem,wen-02-mac} for subsequent developments and \cite{wes-01-pri} for a review).
The use of the MAC scheme in the incompressible case is now standard, and the convergence in this case has been recently tackled in \cite{gallouet2015convergence}. 

The paper is organized as follows.
After recalling the fundamental setting of the problem in the continuous case in Section \ref{2}, we present a simple way (which adapts to the discrete setting) to prove a known preliminary result, namely the convergence (up to a subsequence) of the weak solution of Problem \refe{pbcont_w}-\refe{pressure1} with $ \bm{f}_n $ and $M_n$ (instead of $ \bm{f}$ and $M$) towards a weak solution of Problem \refe{pbcont_w}-\refe{pressure1} (with $M_n \to M$ and $f_n \to f$ weakly in
$L^2(\Omega)^d$ as $\nti$).
Then we proceed in Section \ref{3} to the discretization: we introduce the discrete functional spaces and the definition of the numerical scheme, and state an existence result for this numerical scheme, the proof of which is given in Appendix \ref{existproof}.
The main result of this paper, that is the convergence theorem, is stated in Theorem \ref{mainthm}.
The remaining sections are devoted to the proof of Theorem \ref{mainthm}.
In Section \ref{4}, we derive estimates satisfied by the solutions of the scheme. 
In Section \ref{sec:qdm_and_mass}, we prove the convergence of the numerical scheme in the sense of Theorem \ref{mainthm} toward a weak solution of Problem \eqref{pbcont}-\eqref{pressure1}.

%
% ----------------------------------------------------------------------------------------------------------------------------------------------
%

\section{The continuous problem}\label{2}

\subsection{Definition of weak solution}

In the sequel we explain what we mean by weak solution of Problem \eqref{pbcont_w}--\eqref{pressure1}.
Briefly, if $d=2$ and $\gamma >1$, it is possible to obtain a weak solution $(\bfu,p,\vr)$ of \eqref{pbcont_w}--\eqref{pressure1} in the space $(H^1_0(\Omega))^2 \times L^2(\Omega) \times L^{2 \gamma}(\Omega)$ and to prove the convergence of a sequence of approximate
solutions (up to a subsequence) towards  a weak solution in the sense of  Definition \ref{ldeuxw}. If $d=3$, the problem is much more difficult.
For any $\gamma > 3/2$, a weak solution $(\bfu,p,\vr)$ may be defined (with the extra hypothesis that $\bm{f}$ satisfies $ \curl \bm{f} = \bm{0}$ in the case $ \gamma \in (\frac{3}{2},\frac{5}{3}] $). 
However, this weak solution  belongs to a functional space which depends on $\gamma$.
Indeed, the function $\bfu$ always belongs to $H^1_0(\Omega)^3$, but the function $p$  belongs to $L^2(\Omega)$ only if
$\gamma \ge 3$ (and the function $\vr$  belongs to $L^2(\Omega)$ only if $\gamma \ge 5/3$).  More precisely, for $d=3$ and $\gamma  <3$, we only get an estimate on $p$ in  $ L^{\delta}(\Omega)$, and an estimate on $\vr$ in $L^{\gamma\delta}(\Omega)$, with $\delta=\frac{3(\gamma-1)}{\gamma}$.
Note that for $\gamma=\frac 3 2$, one has $\qbar=\frac{3(\gamma-1)}{\gamma}=1$, and    $\gamma \delta = 3(\gamma-1)=\frac 3 2$, so that the natural spaces are $ p \in L^{1}(\Omega)$ and $\vr \in  L^{\frac 3 2}(\Omega)$. Note that in the case of the compressible Stokes equations, an  $L^{2}$ estimate on the pressure and an $L^{2\gamma}$ estimate on the density are obtained for $d=2$ or $3$ and there is no restriction on $\gamma$ in the sense that we can take $\gamma>1$ (see for instance  \cite{eymard2010convergent} and \cite{eymard2010convergence}).

To be in accordance with the main theorem of this article (see Theorem \ref{mainthm}), we then define the notion of weak solution only  for the case $ \gamma > 3 $ if $d=3$ and $ \gamma >1 $ if $d=2$.
We refer the reader to \cite{novo2002existence} and \cite{straskraba2004introduction} for further informations about the notion of weak solutions and their existence. 
We recall that a bounded Lipschitz domain of $ \R^d$ is a bounded connected open subset of $\R^d$ with a Lipschitz boundary.

In the whole paper, we define the $L^p$ vector norm by:  $\Vert\cdot\Vert_{L^p(\Omega)^d} = \Vert |\cdot |\Vert_{L^p(\Omega)}$, where $|\cdot |$ denotes the Euclidean norm in $\xR^d$.

\begin{df}\label{ldeuxw}Let $d= 2$ or $3$, $\Omega$ be a bounded Lipschitz domain of $ \R^d$ and let  $\bm{f} \in L^2(\Omega)^d$, $M >0$. Let $ \gamma >3 $ if $d=3$ or $\gamma>1$ if $d=2$.
A  weak solution of Problem \eqref{pbcont_w}--\eqref{pressure1} is a function
$(\bfu,p,\vr)  \in  (H^1_0(\Omega))^d \times L^2(\Omega) \times L^{2\gamma}(\Omega)$ satisfying the equations of \eqref{pbcont_w}--\eqref{pressure1} in the following weak sense:

\begin{subequations}
\begin{equation}\label{contf}
	\int_\Omega  \vr \bu \cdot \nabla \varphi \dx= 0, \,  \forall \varphi \in W^{1,\infty} ( \Omega).
\end{equation}
\begin{multline}\label{movf}
	- \int_\Omega \vr \bu \otimes \bu : \nabla \bv \dx +\mu \int_\Omega \nabla \bu : \nabla \bv \dx  +(\mu +\lambda)\int_\Omega \dv \bu \dv \bv \dx \\
 -\int_\Omega p \dv \bv \dx = \int_\Omega \bm{f} \cdot \bv \dx, \;  \, \forall \bv \in C_c^\infty(\Omega)^d.
\end{multline}
\begin{equation}
\vr \ge 0 \textrm{ a.e. in } \Omega, \;  \int_\Omega  \vr \dx =M \textrm{  and } p = \vr^\gamma \textrm{ a.e in } \Omega.
\label{eosf} 
\end{equation}
\label{eq:weak}
\end{subequations}
\end{df}

\begin{rmq} 
Let $(\bu,p,\vr)$ be a weak solution  in the sense of Definition \ref{ldeuxw}. 
Then:
\begin{enumerate}
\item $(\bu,p,\vr)$  satisfies the following inequality (see Step 1 of the proof of Theorem \ref{continuityws})
\begin{equation}\label{ienergief}
\int_\Omega\Big( \mu| \nabla \bu |^2 +(\mu+\lambda)|\dv \bu|^2\Big) \dx\dt\le \int_\Omega \bm{f} \cdot \bu \dx.
\end{equation} 
\item By a density argument, using $\g \ge 3$, one can take $ \bv \in H_0^1(\Omega)^d$ in ($\ref{movf}$).
\end{enumerate}
\end{rmq}

%
% -------------------------------------------------------------------------------------------------------------------------------------------
%

\subsection{Passage to the limit with approximate data}

In order to understand our strategy in the discrete case,
we first prove here the following result (which states the continuity, up to a subsequence, of the weak solution of \refe{pbcont_w}-\refe{pressure1}
with respect to the data). In the following, we set
\begin{equation*} q(d)= \left\{
\begin{array}{l}
  +\infty \ \text{if} \ d=2, \\
 6 \ \text{if} \ d=3.
\end{array}
\right.
\end{equation*}

\begin{thm}\label{continuityws}
Let $ \Omega $ be a bounded Lipschitz domain of $ \R^d$, $d=2$ or $3$. Let $\gamma >1 $ if $d=2$ and $\gamma > 3$ if $d=3$. Let $\bm{f} \in L^2(\Omega)^d$, $M >0$ and $(\bm{f}_n)_\nnn \subset L^2(\Omega)^d$, $(M_n)_\nnn \subset \R_+^\star$ be some sequences satisfying
$\bm{f}_n \tends \bm{f}$ weakly in $(L^2(\Omega))^d$ and $M_n \tends M$. For $\nnn$, let $(\bfu_n,p_n,\vr_n)$ be a weak solution of
\refe{pbcont_w}-\refe{pressure1}, in the sense of Definition \ref{ldeuxw}, with $\bm{f}_n$ and $M_n$ instead of $\bm{f}$ and $M$.

% i.e. a solution to: 
%\begin{subequations}
%\begin{equation}\int \vr_n \bfu_n \cdot \grad \varphi\dx = 0 \textrm{ for all } \varphi \in W^{1,\infty}(\Omega),
%\label{eq:massn} \end{equation}  
%\begin{multline}\label{eq:qdmn}
%	- \int_\Omega \vr_n \bu_n \otimes \bu_n : \nabla \bv \dx +\int_\Omega \mu\nabla \bu_n : \nabla \bv \dx +(\mu +\lambda) \dv \bu_n  \dv \bv \dx \\
% -\int_\Omega p_n \dv \bv \dx = \int_\Omega \bm{f}_n \cdot \bv \dx, \;  \, \forall \bv \in C_c^\infty(\Omega)^d.
%\end{multline}
%\label{eq:nsn} 
%\begin{equation} \vr_n \ge 0  \textrm{ a.e.  in } \Omega, \; \int_\Omega \vr_n \dx = M_n, \;  p_n= \vr_n^\gamma \textrm{ a.e.  in } \Omega.
% \label{eq:poseosn}
%\end{equation}\label{eq:weakn}\end{subequations}

Then, there exists $(\bfu,p,\vr) \in (H^1_0(\Omega))^d \times L^2(\Omega) \times L^{2\gamma}(\Omega)$ such that, up to a subsequence, as $\nti$,
\begin{itemize}
\item $\bfu_n \tends \bfu$ in $(L^q(\Omega))^d$ for $1 \le q < q(d)$ and weakly in $H^1_0(\Omega)^d$,
\item $p_n \tends p$ in $L^q(\Omega)$ for $1 \le q < 2$ and weakly in $L^2(\Omega)$,
\item $\vr_n \tends \vr$ in $L^q(\Omega)$ for $1 \le q < 2 \gamma$ and weakly in $L^{2\gamma}(\Omega)$,
\end{itemize}
and $(\bfu,p,\vr)$ is a weak solution of \refe{pbcont_w}-\refe{pressure1}.
\label{cwrtd}
\end{thm}

\begin{proof}
For the sake of simplicity, we will perform the proof for  $\gamma >3 $ and $d=3$.
The case $d=2$ and $\gamma >1$ is simpler, and the modifications to be done to adapt the proof to the two-dimensional case are mostly due to the fact that Sobolev embeddings differ.

 Let ($\bfu_n, p_n, \vr_n)$ be a weak solution of Problem \eqref{pbcont_w}--\eqref{pressure1} with $f_n$ and $M_n$ instead of $f$ and $M$.

The proof consists in $4$ steps.
In Step~1, we obtain some estimates on $(\bfu_n,p_n,\vr_n)$.
These estimates imply the convergence, in an appropriate sense, of $(\bfu_n,p_n,\vr_n)$ to some $(\bfu,p,\vr)$, up to a subsequence.
Then, it is quite easy to prove that $(\bfu,p,\vr)$ satisfies \refe{contf}, \refe{movf}  and a part of \refe{eosf} (this is Step~2) but it is not easy to prove that $p=\vr^\gamma$ since, using the estimates of Step~1, the convergence of $p_n$ and $\vr_n$ is only weak (and $\gamma \ne 1$).
In Step~3, we prove the convergence of the integral of $p_n \vr_n$ to the integral of $p\vr$.
This allows in Step~4 to obtain the ``strong'' convergence of $\vr_n$ (or $p_n$) and to conclude the proof.

We recall Lemma 2.1 of \cite{eymard2010convergent}, which is crucial for Steps~1 and~3 of the proof. 
This lemma states that if $\vr \in L^{2\gamma}(\Omega)$, $\gamma>1$, $\vr \ge 0$ a.e. in $\Omega$, $\bfu \in (H^1_0(\Omega))^3$ and $(\vr, \bfu)$ satisfies \refe{contf}, then we have:
\begin{equation}
\int_{\Omega} \vr \dv \bu \dx =0
\label{cruun}
\end{equation}
and 
\begin{equation}
\int_{\Omega} \vr^{\gamma}  \dv \bu \dx =0.
\label{crug}
\end{equation}
This result is in fact also true for $\gamma=1$ \cite[Lemma B1]{fettah2012numerical}.
In Step~1 below, we use \refe{crug} (in fact, we only need $\int_{\Omega} \vr^{\gamma} \dv \bu \dx \le 0$ and it is this weaker result
which will be adapted and used for the approximate solution obtained by a numerical scheme).
In Step~3, we use \refe{cruun}.

\medskip
 
{\bf Step 1. Estimates.} We recall that ($\bfu_n, p_n, \vr_n)$ satisfies \eqref{eq:weak} with $f_n$ and $M_n$.% instead of $f$ and $M$.

 {\it 1.a Estimate on the velocity.} Taking $\bfu_n$ as a test function in \eqref{movf}, we get:
\begin{multline*}
\mu \int_\Omega \grad \bfu_n : \grad \bfu_n\dx + (\mu+\lambda)\int_\Omega (\dv \bu_n)^2 \dx  - {\int_\Omega \vr_n \bfu_n \otimes \bfu_n : \grad \bfu_n\dx} \\ -  {\int_\Omega p_n \dv \bfu_n \dx} = \int_\Omega \bm{f}_n \cdot \bfu_n\dx.
\end{multline*}
Note that, since $\gamma >3$, we have $ \vr_n \bfu_n \otimes \bfu_n \in L^2(\Omega)^{3 \times 3} $, and, by density of $C_c^\infty(\Omega)^d$ in $L^2(\Omega)^d$, $\bfu_n$ is indeed an admissible test function.
But $p_n=\vr_n^\gamma$ a.e. in $\Omega$ and $\div(\vr_n \bfu_n)=0$ (in the sense of \eqref{contf}), then using \eqref{crug}
(with $\vr_n$ and $\bfu_n$)
\begin{equation*}
\int_\Omega p_n \div \bu_n \, \dx=0.
\end{equation*}
Again thanks to the mass equation \eqref{contf}, and to the fact that $\vr_n \in L^{2\gamma} (\Omega) \subset L^6(\Omega)$ a straightforward computation gives
\begin{equation*}
{\int_\Omega \vr_n \bfu_n \otimes \bfu_n : \grad \bfu_n \, \dx=0}.
\end{equation*}
Hence, there exists $C_1$, only depending on the $L^2-$bound of $(\bm{f}_n)_\nnn$,  on $\Omega$ and on $\mu$, such that:
\begin{equation}
\norm{\bfu_n}{(H^1_0(\Omega))^3} \le C_1.
\label{estu}
\end{equation}

{\it  1.b Estimate on the pressure.} 
In order to obtain an estimate on $p_n$ in $L^2(\Omega)$, we now use the two following lemmas. The first one is due to Bogovski,  see e.g. \cite[Section 3.3]{straskraba2004introduction} or \cite[Theorem 10.1]{feireisl2009singular} for a proof.

\begin{lm}
Let $\Omega$ be a bounded Lipschitz domain of $\R^d$ $ (d\ge1)$.
Let $r \in (1,+\infty)$.
Let $q \in L^{r}(\Omega)$ such that $\int_{\Omega} q\dx =0$. Then, there exists $\bv  \in (W^{1,r}_0(\Omega))^d$ such that $\div \bv =q$ a.e. in $\Omega$ and $\norm{\bv}{(W^{1,r}_0(\Omega))^d}  \le C_2 \norm{q}{L^{r}(\Omega)}$ with $C_2$ depending only on $\Omega$ and $r$.
\label{lem:bogos}
\end{lm}

The following lemma is a straightforward consequence of \cite[Lemma 5.4]{fettah2014existence}.

\begin{lm}\label{estl2} 
Let $\Omega$ be a bounded Lipschitz domain of $\R^d$ $ (d\ge1)$  and $p \in L^2(\Omega) $ such that  $p \ge 0 $ a.e in $\Omega$.
We assume that there exist $a > 0$, $b, c \in \R$ and $ r \in (0,1)$ such that
\begin{equation*}
\left\{
\begin{array}{l}
  \| p -m(p) \|_{L^2(\Omega)} \le a \normexp{p}{L^2(\Omega)}{r} + b,
  \\ \displaystyle
  \int_\Omega p^r \dx \le c,
\end{array}
\right.
\end{equation*}
where $m(p)=\frac{1}{|\Omega|} \int_\Omega p \dx$ is the mean value of $p$. Then, there exists $C$ only depending on $\Omega, a, b, c$ and $r$ such that 
$ \| p \|_{L^2(\Omega)} \le C.$
\end{lm}

Let $m_n=\frac 1 {\vert \Omega \vert} \int_\Omega p_n\dx$; thanks to Lemma \ref{lem:bogos} with $r =2$, there exists $\bv_n \in H^1_0(\Omega)^3$ such that $\div \bv_n =p_n-m_n$ and 
\begin{equation}
\norm{\bv_n}{(H^{1}_0(\Omega))^3}  \le C_2 \norm{p_n-m_n}{L^{2}(\Omega)}.   \label{eq:necas}                                                                                                                                                                          \end{equation}
Taking $\bv_n$ as a test function in  \eqref{movf} yields:
\begin{multline}
\mu \int_\Omega \grad \bfu_n : \grad \bv_n\dx+(\mu+\lambda)\int_\Omega \dv\bfu_n \dv\bv_n\dx -\int_\Omega \vr_n \bfu_n \otimes \bfu_n : \grad \bv_n \dx \\ - \int_\Omega p_n \dv \bv_n\dx = \int _\Omega \bm{f}_n  \cdot \bv_n\dx. \label{eq:biarritz1}
\end{multline}
Since  $\displaystyle \int_\Omega \dv \bu_n \dx = \int_\Omega \dv\bv_n\dx=0$, we get:
\[
\int_\Omega (p_n-m_n)^2\dx = \int_\Omega \bigl( -\bm{f}_n  \cdot \bv_n + \mu \grad \bfu_n : \grad \bv_n +(\mu+\lambda)p_n \dv \bfu_n -\vr_n \bfu_n \otimes \bfu_n : \grad \bv_n\bigr)\dx.
\]
Since $\norm{\bfu_n}{(H^1_0(\Omega))^3} \le C_1$ and $H^1_0(\Omega)$ is continuously embedded in $L^6(\Omega)$, we get that:
\begin{equation}
\int_\Omega \vr_n \bfu_n \otimes \bfu_n : \grad v_n\dx \le \norm{\vr_n}{L^6(\Omega)} \normd{\bfu_n}{L^6(\Omega)^3} \norm{\bv_n}{(H^1(\Omega))^3}.
\label{eq:biarritz2}
\end{equation}
From \eqref{eq:biarritz1},  \eqref{eq:biarritz2} and \eqref{eq:necas}, since $2 \gamma \ge 6$ and $p_n=\vr_n^\gamma$, we get:
\[
\norm{p_n-m_n}{L^2(\Omega)} \le C_3\,(1+ \norm{\vr_n}{L^6(\Omega)}) \le C_4\, (1+ \norm{\vr_n}{L^{2\gamma}(\Omega)})
\le C_4\, (1+  \normexp{p_n}{L^2(\Omega)}{1/\gamma}).
\]
Since $\int_\Omega p_n^{1/\gamma} \dx = \int_\Omega \vr_n\dx \le \sup\{M_k, k \in \N\}$, we get from Lemma \ref{estl2} that
$
\norm{p_n}{L^2(\Omega)} \le C_5,
$
where $C_5$ depends only on the $L^2-$bound on $(\bm{f}_n)_\nnn$, the bound on $(M_n)_\nnn$, $\gamma$, $\mu$, $\lambda$ and $\Omega$.
Thanks to the equation of state, we have $p_n=\vr_n^\gamma$ a.e. in $\Omega$, and therefore
$
\norm{\vr_n}{L^{2\gamma} (\Omega)} \le C_6=C_5^{1/ \gamma}.
$

\vspace{1cm}

 {\bf Step 2. Passing to the limit on the equations \refe{contf}, \refe{movf} and a part of \refe{eosf}}.

The estimates obtained in Step~1 yield that, up to a subsequence, as $\nti$:
\[
\begin{array}{l}\displaystyle
\bfu_n \tends u \textrm{ in } L^q(\Omega)^3 \ \textrm{for any} \ 1 \le q<6 \textrm{ and weakly in }  H^1_0(\Omega)^3,
\\[1ex] \displaystyle
p_n \tends p \textrm{ weakly in } L^2(\Omega),
\\[1ex] \displaystyle
\vr_n \tends \vr \textrm{ weakly in } L^{2\gamma}(\Omega).
\end{array}
\]
%The weak momentum equation for $(\bfu_n, p_n, \vr_n)$ reads: 
%\begin{multline*}
%\mu \int_\Omega { \grad \bfu_n} : \grad \bv \, \dx+(\mu+\lambda)\int_\Omega \dv \bfu_n \dv \bv \dx  -\int_\Omega {  { \vr_n} {\bfu_n} \otimes  {\bfu_n} }: \grad \bv \, \dx\\ - \int_\Omega  {p_n }\dv \bv \, \dx
% =\int_\Omega  {\bm{f}_n }\cdot \bv \,\dx, \forall v \in C^\infty_c(\Omega)^3.
%\end{multline*}
Since ${\vr_n \tends \vr}$ weakly in $L^{2 \gamma}(\Omega)$, with $2 \gamma > 6>\frac 3 2$, and ${\bfu_n \tends \bfu}$ in $L^q(\Omega)$ for all $q <6$ (and $\frac 2 3 + \frac 1 6 + \frac 1 6=1$), we have that ${\vr_n \bfu_n \otimes \bfu_n \tends \vr \bfu \otimes \bfu}$ weakly  in $L^1(\Omega)$.
Moreover, $  { \grad \bfu_n \tends \grad \bfu}$ weakly in $L^2(\Omega)^3$, ${p_n \tends p}$ weakly in $L^{2}(\Omega)$
and ${\bm{f}_n \tends \bm{f}}$ weakly in $L^{2}(\Omega)^3$.
Therefore, passing to the limit  in \refe{movf} (the weak momentum equation)  for $(\bfu_n, p_n, \vr_n)$, we obtain \refe{movf} for $(\bfu,p,\vr)$.

%The weak mass balance equation $(\bfu_n, \vr_n)$ reads: 
%\[
%\int_\Omega \vr_n \bfu_n \cdot \grad \varphi \dx = 0, \forall \varphi \in W^{1,\infty}(\Omega).
%\]
Since $\vr_n \tends \vr$ weakly in $L^{2 \gamma}(\Omega)$, with $2 \gamma > \frac 6 5 $
and $\bfu_n \tends \bfu$ in $L^q(\Omega)$ for all $q <6$, we get that $\vr_n \bfu_n \tends \vr \bfu$ weakly in $L^1(\Omega)$.
Then passing to the limit on \refe{contf} (the weak mass balance) for $(\bfu_n, \vr_n)$, we obtain \refe{contf} for $(\bfu,\vr)$.

%
%This gives
%\begin{equation*}
%\int_\Omega \vr \bfu \cdot \grad \varphi = 0, \forall \varphi \in W^{1,\infty}(\Omega).
%\end{equation*}
%

The weak convergence of $\vr_n$ to $\vr$ and the fact that $\vr_n \ge 0$ a.e.  in $\Omega$ gives that
$\vr \ge 0$ a.e.  in $\Omega$ (indeed, taking $\psi = 1_{\vr <0}$ as test function  gives $\int_{\Omega} \vr \psi\dx
=\lim_{\nti} \int_{\Omega} \vr_n \psi\dx \ge 0$, which proves that $\vr \psi=0$ a.e.).
The weak convergence of $\vr_n$ to $\vr$ also gives   (taking $\psi=1$ as test function) that $\int_{\Omega} \vr\dx =M$. 
Therefore, $(\bu,p,\vr)$ is a weak solution of the momentum equation and of the mass balance equation satisfying $ \vr \ge 0$ a.e in $\Omega$ and $ \int_\Omega \vr \dx =M$.
Hence Theorem \ref{continuityws} is proved except for the fact that $p= \vr^\gamma$ a.e.  in $\Omega$. This is the objective of the last two steps, where we also prove a ``strong'' convergence of $\vr_n$ and $p_n$.
We need to prove that $p=\vr^\gamma \textrm{ in } \Omega$, even though we only have a weak convergence of $p_n$ and $\vr_n$, and $\gamma >1$.
The idea (for $d=2$ or $d=3$, $\gamma > 3$) is to prove $\int_\Omega  p_n \vr_n  \tends \int_\Omega p \vr $ and deduce the a.e. convergence (of $ p_n$ and $ \vr_n$) and $p=\vr^\gamma$.

{\bf Step 3. Proving the convergence of the effective viscous flux and $\displaystyle \int_{\Omega}  \vr_n p_n\dx  \tends \int_{\Omega} \vr p\dx$.}

Since the sequence $(\vr_n)_\nnn$ is bounded in $L^2(\Omega)$,
The result of \cite[Lemma B.8]{eymard2010convergent} gives the existence of a bounded sequence $(\bv_n)_\nnn$ in $H^1(\Omega)^3$ such that $\div \bv_n =\vr_n$ and $\curl \bv_n =0$.
It is possible to assume (up to a subsequence) that $\bv_n \tends v$ in $L^2(\Omega)^3$ and weakly in $H^1(\Omega)^3$.
Passing to the limit in the preceding equations gives $\div \bv =\vr$ and $\curl \bv =0$.

Let $\varphi \in C^\infty_c(\Omega)$ (so that $\varphi \bv_n \in H^1_0(\Omega)^3$). 
Taking $\bv=\varphi \bv_n $ in  the weak momentum equation \eqref{movf} written for $(\bu_n,p_n,\vr_n)$) leads to:
\begin{multline}
\mu \int_{\Omega} \grad \bfu_n : \grad (\varphi \bv_n) \,\dx +(\mu+\lambda) \int_\Omega \dv \bfu_n \dv( \varphi \bv_n) \dx
  - \int_{\Omega} p_n \div (\varphi \bv_n)\, \dx \\ =  \int_\Omega \vr_n \bfu_n \otimes \bfu_n : \grad (\varphi \bv_n)\, \dx
  +\int _{\Omega} \bm{f}_n \cdot (\varphi \bv_n) \,\dx.
  \label{c-weakmom}
\end{multline}
The choice of $\bv_n$ gives $\div(\varphi \bv_n)= \varphi \vr_n  + \bv_n \cdot \grad \varphi$ and $\curl \varphi \bv_n  =L(\varphi)\bv_n$, where $L(\varphi)$ is a matrix with entries involving the first order derivatives of $\varphi$.
Noting that 
\begin{equation}\label{graddivrot}
\int_\Omega \nabla \bar \bu : \nabla \bar \bv \dx = \int_\Omega \dv \bar \bu \dv \bar \bv \dx + \int_\Omega \curl \bar \bu \cdot \curl \bar \bv \dx, \textrm{ for all } (\bar \bu,\bar \bv) \in H_0^1(\Omega)^3,
\end{equation}
the  equality \eqref{c-weakmom} leads to:
\begin{multline*}
 \int_{\Omega} \Big((2\mu+\lambda)\dv \bu_n - p_n\Big) \vr_n \varphi \dx +  \int_{\Omega}\Big((2\mu+\lambda) \dv \bu_n - p_n\Big) \bv_n \cdot \grad \varphi \dx \\ +\mu \int \curl \bfu_n  \cdot  L(\varphi) \bv_n  \dx =  \int_\Omega \vr_n \bfu_n \otimes \bfu_n : \grad ( \varphi \bv_n)\, \dx+
\int _{\Omega} \bm{f}_n \cdot (\varphi \bv_n) \,\dx.
\end{multline*}
Thanks to the weak  convergence of $\bfu_n$ in $H^1_0(\Omega)^d$ to $\bfu$, the weak convergence of $p_n$ in $L^2(\Omega)$ to $p$, the weak convergence of $\bm{f}_n$ in $L^2(\Omega)$ to $\bm{f}$  and the convergence of $\bv_n$  in $L^2(\Omega)^d$ to $\bv$, we obtain:
\begin{multline}
\lim_{\nti} \int_{\Omega}  \Big(\Big((2\mu+\lambda)\div \bfu_n  - p_n\Big) \vr_n \varphi - \vr_n \bfu_n \otimes \bfu_n :  \grad (\varphi v_n) \Big)\, \dx= \\ \int _{\Omega} \bm{f} \cdot (\varphi \bv) \dx + \int_{\Omega} \Big(p-(2\mu+\lambda)\dv \bu\Big) \bv \cdot \grad \varphi\dx - \mu  \int _{\Omega)}\curl \bfu  \cdot  L(\varphi)\bv \dx. 
\label{cpc}
\end{multline}
But, thanks to the weak momentum equation \refe{movf} for $(\bu,p,\vr)$, we have
\begin{multline*}
\mu \int_{\Omega} \grad \bfu : \grad (\varphi \bv )\,\dx +(\mu+\lambda) \int_\Omega \dv \bfu \dv (\varphi \bv) \dx  - \int_{\Omega} p \div (\varphi \bv )\, \dx \\ =  \int_{\Omega} \vr  \bfu  \otimes \bfu  : \grad ( \varphi \bv )\, \dx + \int _{\Omega} \bm{f} \cdot(\varphi \bv )\,\dx,
\end{multline*}
or equivalently, thanks to \eqref{graddivrot}:
\begin{multline*}
\int_{\Omega} \Big((2\mu+\lambda) \dv \bu -p \Big) \dv (\varphi \bv) \dx  +\mu \int_{\Omega} \curl  \bfu  \cdot \curl(\varphi \bv )\,\dx \\ =  \int_\Omega \vr  \bfu  \otimes \bfu  : \grad (\varphi \bv  )\, \dx+ \int _{\Omega} f \cdot (\varphi \bv ) \,\dx.
\end{multline*}
Since $\div \bv =\vr$ and  $\curl \bv =0$, we obtain:
\begin{multline*}
\int_{\Omega} \Big( (2\mu+\lambda)\dv \bu-p\Big) \vr \vphi\dx - \int_\Omega \vr  \bfu  \otimes \bfu  : \grad (\varphi \bv)  \, \dx=  \\  \int_{\Omega} \bm{f} \cdot (\varphi \bv)\dx + \int_{\Omega} (p- \Big(2\mu+\lambda)\dv \bu \Big) \bv \cdot \grad \varphi\dx - \mu  \int_{\Omega} \curl \bfu  \cdot L(\varphi)\bv\dx.
\end{multline*}
Let us assume momentarily that:
\begin{equation}
\int_{\Omega}    \vr_n \bfu_n \otimes \bfu_n :  \grad (\varphi \bv_n) \, \dx  \to \int_\Omega \vr  \bfu  \otimes \bfu  : \grad (\varphi \bv)  \, \dx \mbox{ as } \nti.
\label{ass:kk}
\end{equation}
We then obtain thanks to  \refe{cpc}:
\begin{equation}
\lim_{\nti} \int_{\Omega} \Big(p_n -(2\mu+\lambda) \dv \bu_n\Big) \vr_n \varphi \dx= \int_{\Omega} \Big(p -(2\mu+\lambda) \dv \bu \Big) \vr \varphi\dx.
\label{palnl}
\end{equation}
The quantity $p-(\lambda+2\mu)\dv \bu $ is usually  called the effective viscous flux. 
This quantity enjoys many remarkable properties for which we refer to Hoff \cite{hoff1995strong}, Lions \cite{lions1998mathematical}, or Serre \cite{serre1991variations}. 
Note that this quantity is the amplitude of the normal viscous stress augmented by the hydrostatic pressure $p$, that is, the ``real'' pressure acting on a volume element of the fluid.
In \refe{palnl}, the function $\vphi$ is an arbitrary element of $C^\infty_c(\Omega)$. 
Then as in \cite{eymard2010convergent}, we remark that it is possible to take $\vphi=1$ in \refe{palnl}, thanks to the fact that $(p_n - (2\mu+\lambda)\dv \bu_n) \vr_n \in L^r(\Omega)$ for some $r >1$ (see \cite[Lemma B.2]{eymard2010convergent}). 

Using \refe{cruun}, which holds by  \cite[Lemma 2.1]{eymard2010convergent} thanks to the fact that $\div(\vr_n \bfu_n) =\div (\vr \bfu)=0$ (in the sense of \eqref{contf}), we have
$\int_{\Omega} \vr_n \div \bfu_n \dx=\int_{\Omega} \vr \ \div \bfu \dx=0.$
Therefore, \refe{palnl} yields:
\begin{equation}
\lim_{\nti} \int_{\Omega} p_n  \vr_n\dx= \int_{\Omega} p \vr\dx.
\label{palnlff}
\end{equation}
\begin{rmq}
The equality in \refe{palnlff} is not necessary in Step 4; in fact, it is sufficient to have
$ \liminf_\nti \int_{\Omega} p_n \vr_n\dx $ $ \le  \int_{\Omega} p \vr\dx.$
Then, instead of  $\int_{\Omega} \vr_n \div \bfu_n \dx=0$,  it is sufficient to have
$ \liminf_\nti \int_{\Omega} \vr_n \div \bfu_n \dx \le 0.$
This will be the case in the framework of an approximation by a numerical scheme.
\end{rmq}

In order to conclude Step 3, it remains to show \eqref{ass:kk}.

We remark that, since $\div(\vr_n \bfu_n)=0$ and $(\vr_n,\bfu_n) \in L^6(\Omega)\times H_0^1(\Omega)^3 $,
\begin{equation}
 \int_\Omega \vr_n \bfu_n \otimes \bfu_n : \grad (\varphi \bv_n) \,\dx = -\int_\Omega (\vr_n \bfu_n \cdot \grad )\bfu_n \cdot(\varphi \bv_n) \,\dx.
 \label{conv-rewrite}
\end{equation}

The sequence $((\vr_n \bfu_n \cdot \grad )\bfu_n)_\nnn$ is bounded in $L^r(\Omega)^3$, with
$\frac 1 r = \frac 1 2 + \frac 1 6 + \frac 1 {2\gamma}$. 
Since $\gamma > 3$, we have $r >\frac 6 5$.
Then, up to a subsequence, $(\vr_n \bfu_n \cdot \grad )\bfu_n$ tends to some function $G$ weakly in $L^r(\Omega)^3$.
Since $\bv_n \tends \bv$ in $L^s(\Omega)^3$ for all $s<6$ and therefore for $s = \frac{r}{r-1}$,
we deduce that:
\[
	\int_\Omega (\vr_n \bfu_n \cdot \grad )\bfu_n \cdot (\varphi \bv_n)\dx \tends \int_\Omega G \cdot (\varphi \bv)\dx.
\]
Moreover, for a fixed $\bm{w} \in H^1_0(\Omega)^3$,
\begin{equation*}
\int_\Omega (\vr_n \bfu_n \cdot \grad )\bfu_n \cdot \bm{w}\dx= - \int_\Omega \vr_n \bfu_n\otimes \bfu_n : \grad \bm{w} \dx \tends   - \int_\Omega \vr \bfu \otimes \bfu : \grad \bm{w} \dx.
\end{equation*}
But, since  $ \dv(\vr \bu) =0 $ and $ (\vr,\bfu) \in L^6(\Omega) \times H_0^1(\Omega)^3 $, we have
\begin{equation*}
- \int_\Omega \vr \bfu \otimes \bfu : \grad \bm{w} \dx = \int_\Omega (\vr \bfu \cdot \grad )\bfu \cdot \bm{w} \,\dx.
\end{equation*}
We thus get that $G=(\vr \bfu \cdot \grad) \bfu$, which concludes the proof of ($\ref{ass:kk}$).

\medskip

{\bf Step 4. Passing to the limit on the EOS and ``strong'' convergence of $\vr_n$ and $p_n$.}
The end of the proof is exactly the same as Step~4 of \cite[Proof of Theorem 2.2]{eymard2010convergent}; it is reproduced here for the sake of completeness.
For $\nnn$, let $G_n= (\vr_n^\gamma - \vr^\gamma)(\vr_n - \vr)$.
For all $\nnn$, the function $G_n$ belongs to $L^1(\Omega)$ and $G_n \ge 0$ a.e. in $\Omega$.
Futhermore $G_n=(p_n -\vr^\gamma)(\vr_n - \vr)= p_n \vr_n-p_n \vr - \vr^\gamma \vr_n  + \vr^\gamma \vr$ and
$\displaystyle
\int_\Omega G_n \dx = \int_\Omega p_n \vr_n \dx - \int_\Omega p_n \vr \dx  - \int_\Omega  \vr^\gamma \vr_n \dx + \int_\Omega \vr^\gamma \vr \dx.
$

Using the weak convergence in $L^2(\Omega)$ of $p_n$ to $p$ and of $\vr_n$ to $\vr$, the fact that $\vr, \vr^\gamma \in L^2(\Omega)$
and \eqref{palnlff} gives
$
\lim_\nti \int_\Omega G_n \dx =0,
$
that is $G_n \tends 0$ in $L^1(\Omega)$.
Then, up to a subsequence, we have $G_n \tends 0$ a.e. in $\Omega$.
Since $y \mapsto y^\gamma$ is an increasing function on $\xR_+$, we deduce that $\vr_n \tends \vr$ a.e., as $\nti$.
Then, we also have $p_n = \vr_n^\gamma \tends \vr^\gamma$ a.e..
Since $(\vr_n)_\nnn$ is bounded in $L^{2 \gamma}(\Omega)$ and $(p_n)_\nnn$ is bounded in $L^2(\Omega)$, we obtain, as $\nti$:
\[
\begin{array}{l} \displaystyle
\vr_n \tends \vr \textrm{ in } L^q(\Omega) \textrm{ for all } 1 \le q < 2\gamma,
\\[1ex] \displaystyle
p_n \tends \vr^\gamma \textrm{ in } L^q(\Omega) \textrm{ for all } 1 \le q < 2.
\end{array}
\]
Since we already know that $p_n \tends p$ weakly in $L^2(\Omega)$, we necessarily have (by uniqueness of the weak limit in $L^q(\Omega)$) that $p=\vr^\gamma$ a.e. in $\Omega$.
The proof of Theorem~\ref{cwrtd} is now complete.
%By a monotonicity argument, we conclude, exactly as in \cite[Proof of Theorem 2.2]{eymard2010convergent}, that 
%$$
%\vr_n \tends \vr \textrm{ in } L^q(\Omega) \textrm{ for all } 1 \le q < 2\gamma.
%$$
%$$
%p_n \tends \vr^\gamma \textrm{ in } L^q(\Omega) \textrm{ for all } 1 \le q < 2, 
%$$
%as $\nti$, and therefore that  $p=\vr^\gamma$ a.e. in $\Omega$.  
\end{proof}

%
% -------------------------------------------------------------------------------------------------------------------------------------------------
%

\section{The numerical scheme}\label{3}

\subsection{Mesh and discrete spaces}

We will now assume that the bounded domain $\Omega$ is \textit{MAC compatible} in the sense that $\bar \Omega$ is a finite union of (closed) rectangles ($d=2$) or (closed) orthogonal parallelepipeds ($d=3$) and, without loss of generality, we assume that the  edges (or faces) of these rectangles (or parallelepipeds) are orthogonal to the canonical basis vectors, denoted by $(\bfe_1, \ldots, \bfe_d)$.

\begin{df}[MAC grid]
\label{def:MACgrid}
A discretization of a MAC compatible bounded domain $\Omega$ with a MAC grid is defined by $\mathcal{D} = (\mesh, \edges)$,
where:
\begin{list}{-}{\itemsep=0.ex \topsep=0.5ex \leftmargin=1.cm \labelwidth=0.7cm \labelsep=0.3cm \itemindent=0.cm}
\item $\mesh$ stands for the primal grid, and consists in a regular structured partition of $\Omega$ in possibly non uniform rectangles ($d=2$) or rectangular parallelepipeds ($d=3$).
A generic cell of this grid is denoted by $K$, and its mass center by $\bfx_K$.
The scalar unknowns, namely the density and the pressure, are associated to this mesh, and $\mesh$ is also sometimes referred as "the pressure mesh".

\item The set of all faces of the mesh is denoted by $\edges$; we have $\edges= \edgesint \cup \edgesext$, where $\edgesint$ (resp. $\edgesext$) are the edges of $\edges$ that lie in the interior (resp. on the boundary) of the domain.  
The set of faces that are orthogonal to the $i^{th}$ unit vector $\bfe_{i}$ of the canonical basis of $\mathbb{R}^d$ is denoted by $\edgesi$, for $i = 1,\ldots,d$.
We then have $\edgesi= \edgesinti \cup \edgesexti$, where $\edgesinti$  (resp. $\edgesexti$) are the edges of $\edgesi$ that lie in the interior (resp. on the boundary) of the domain.

For each $\edge\in\edges$, we write that $\edge = K \vert L$ if $\edge = \partial K \cap \partial L$.
A dual cell $D_{\edge}$ associated to a face $\edge \in\edges$ is defined as follows:  
\begin{list}{$\ast$}{\itemsep=0.ex \topsep=0.ex  \leftmargin=1.cm \labelwidth=0.7cm \labelsep=0.1cm \itemindent=0.cm}
\item if $\edge=K|L \in \edgesint$ then  $D_{\edge} = D_{ K,\edge}\cup D_{ L,\edge}$, where $D_{ K,\edge}$ 
(resp. $D_{ L,\edge}$) is the half-part of $K$ (resp. $L$) adjacent to $\edge$ (see Fig. \ref{fig:mesh} for the two-dimensional case)~; 
\item if $\edge \in \edgesext$ is adjacent to the cell $K$, then $D_\edge=D_{ K,\edge}$.
\end{list}
We obtain $d$ partitions of the computational domain $\Omega$ as follows:
\[
\Omega = \cup_{\edge \in \edgesi} D_\edge,\quad 1 \leq i \leq d,
\]
and the $i^{th}$ of these partitions is called $i^{th}$ dual mesh, and is associated to the $i^{th}$ velocity component, in a sense which is precised below.
The set of the faces of the $i^{th}$ dual mesh is denoted by $\edgesdi$ and is decomposed into the internal and boundary edges: $\edgesdi = \edgesdinti\cup \edgesdexti$.
The dual face separating two dual cells $D_\edge$ and $D_{\edge'}$ is denoted by $\edged=\edge|\edge'$.
\end{list}
\end{df}

To define the scheme, we need some additional notations.
The set of faces of a primal cell $K$ and a dual cell $D_\edge$ are denoted by $\edgesK$ and $\edgesd(D_\edge)$ respectively.
%For $\edge \in \edgesK$, $\bn_{\edge,K} $ stands for the unit normal vector to $\edge$ outward $K$.
For $\edge \in \edges$, we denote by $\bfx_\edge$ the mass center of $\edge$.

In some cases, we need to specify the orientation of a geometrical quantity with respect to the axis:
\begin{list}{-}{\itemsep=0.ex \topsep=0.5ex \leftmargin=1.cm \labelwidth=0.7cm \labelsep=0.3cm \itemindent=0.cm}
\item a primal cell $K$ will be denoted $K = [\overrightarrow{\edge \edge'}]$ if 
there exists $i\in\llbracket 1, d\rrbracket$ and $\edge, \edge' \in \edgesi \cap \edges(K)$
such that $(\bfx_{\edge'} - \bfx_{\edge}) \cdot \bfe_i >0$;
\item we write $\edge =\overrightarrow{K \vert L}$ if  $\edge \in\edgesi$ and $\overrightarrow{\bfx_{K}\!\bfx_{L}}\cdot \bfe_{i}>0$ for some $i \in \llbracket 1,d \rrbracket$;
\item the dual face $\edged$ separating $D_\edge$ and $D_{\edge'}$ is written $\edged = \overrightarrow{\edge\!\vert{\edge'}}$ if $\overrightarrow{\bfx_{\edge}\!\bfx_{\edge'}}\cdot \bfe_{i}>0$ for some $i \in\llbracket 1, d\rrbracket$.
\end{list}
For the definition of the discrete momentum diffusion operator, we associate to any dual face $\edged$ a distance $d_\edged$ as sketched in Figure \ref{fig:mesh}. 
For a dual face $\edged \in \edgesd(D_\edge), \edge \in \edgesi$, $i \in {\llbracket 1, d\rrbracket}$, the distance $d_\edged$ is defined by:
 \begin{align}
  d_{\edged} =  \begin{cases}
    d(\bfx_{\edge},\bfx_{\edge'}) & \mbox{if} \  \edged = \edge\! \vert\edge'\in \edgesdinti,\\[1ex]
    d(\bfx_{\edge},\edged) & \mbox{if} \ \edged \in \edgesdexti \cap \tilde{\edges}(D_{\edge})
  \end{cases}
  \label{depsilon}
 \end{align}
where $d(\cdot,\cdot)$ denotes the Euclidean distance in $\xR^d$. 
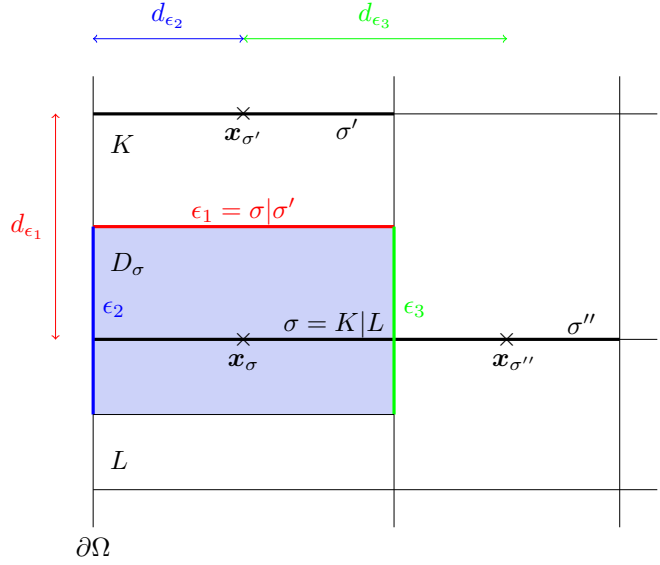
\begin{figure}[hbt] 
\centering
\begin{tikzpicture}
\fill[green!12!blue!20!white] (0.5,1.5)--(4.5,1.5)--(4.5,4)--(0.5,4)--cycle;
\path node at (0.6,3.5) [anchor= west]{$D_\edge$};

% primal mesh:
\draw[very thin] (0.5,0.5)--(8.,0.5);
\draw[very thin] (0.5,2.5)--(8,2.5);
\draw[very thin] (0.5,5.5)--(8,5.5);
\draw[very thin] (0.5,0.)--(0.5,6.);
\draw[very thin] (4.5,0.)--(4.5,6.);
\draw[very thin] (7.5,0.)--(7.5,6.);
\draw[very thick] (0.5,2.5)--(4.5,2.5);
\draw[very thick] (0.5,5.5)--(4.5,5.5);
\draw[very thick] (4.5,2.5)--(7.5,2.5);

% dual edges:
\draw[very thick, red] (0.5,4)--(4.5,4);
\draw[very thick, green] (4.5,1.5)--(4.5,4);
\draw[very thick, blue] (0.5,1.5)--(0.5,4);
\draw[very thin] (0.5,1.5)--(4.5,1.5);

\path node at (0.6,5.1) [anchor= west]{$K$};
\path node at (0.6,0.9) [anchor= west]{$L$};
\path node at (3.7,2.7) {$\edge=K|L$};
\path node at (7,2.7) {$\edge''$};
\path node at (2.5,2.5) {$\times$};
\path node at (2.5,5.5) {$\times$};
\path node at (6,2.5) {$\times$};

%------------------moi
\path node at (2.5,5.2) {$\bfx_{\edge'}$};
\path node at (2.5,2.2) {$\bfx_{\edge}$};
\path node at (6.1,2.2) {$\bfx_{\edge''}$};
%---------------------------
\path node at (0.5,2.9) [anchor= west]{\textcolor{blue}{$\edged_2$}};
\path node at (4.5,2.9) [anchor= west]{\textcolor{green}{$\edged_3$}};
\path node at (3.88,5.3) {$\edge'$};
\path node at (2.5,4.2) {\textcolor{red}{$\edged_1=\edge|\edge'$}};
\path node at (0.5,-0.02) [anchor= north]{$\partial\Omega$};
\draw[very thin, green, <->] (2.5,6.5)--(6,6.5); \path node at (4.25,6.52) [anchor= south]
{\textcolor{green}{$d_{\edged_3}$}};
\draw[very thin, blue, <->] (0.5,6.5)--(2.5,6.5);\path node at (1.5,6.52) [anchor= south]
{\textcolor{blue}{$d_{\edged_2}$}};
\draw[very thin, red, <->] (0,2.5)--(0,5.5); \path node at (-0.02,4) [anchor= east]
{\textcolor{red}{$d_{\edged_1}$}};

\end{tikzpicture}
\caption{Notations for control volumes and dual cells (for the second component of the velocity).}
\label{fig:mesh}
\end{figure}
We also define the size of the mesh by 
 $
 h_\mesh=\max\{\diam(K),  K\in\mesh\}.
$
The regularity of $\eta_{\mesh} $ of the mesh is defined by
\begin{equation}
\eta_{\mesh}  =\frac 1 {h _\mesh} \min_{K \in \mesh}\ \min_{1 \le i \le d}
\ \{ d(\bfx_\edge, \bfx_{\edge'}), ~\edge,\edge' \in \E^{(i)}(K) \}.
\label{regmesh} \end{equation}
In other words, $ \eta_\mesh $ is such that
%$
%\eta_\mesh h_\mesh \le d(\bfx_\edge, \bfx_{\edge'}) \le h_\mesh, \ \forall \edge,\edge' \in \E^{(i)}(K), \forall  i=1,...,d, \, \forall K \in \mesh.
%$
\begin{equation*}
\eta_\mesh h_\mesh \le d(\bfx_\edge, \bfx_{\edge'}) \le h_\mesh, \ \forall \edge,\edge' \in \E^{(i)}(K), \forall  i=1,...,d, \, \forall K \in \mesh.
\end{equation*}

The discrete velocity unknowns are associated to the velocity cells and are denoted by $(u_{\edge})_{\edge\in\edgesi}$ for each component $u_i$ of the discrete velocity, $1\le i \le d$, while the discrete density and pressure unknowns are associated to the primal cells and are respectively denoted by $(\vr_{K})_{K\in\mesh}$ and $(p_{K})_{K\in\mesh}$.
\begin{df}[Discrete spaces]
\label{discretespace}
Let $\mathcal{D}=(\mesh,\edges)$ be a MAC grid in the sense of Definition \ref{def:MACgrid}. 
The discrete density and pressure space $L_{\mesh}$ is   the set of piecewise constant functions over the grid cells $K$ of $\mesh$, and  the discrete $i^{th}$ velocity space $\Hmeshi$ is the set of piecewise constant functions over   the grid cells $D_\edge~,\edge\in\edgesi$.
The Dirichlet boundary conditions ($\ref{ci2}$) are partly incorporated in the definition of the velocity spaces by introducing 
\[
 \Hmeshizero=\Bigl\{u\in\Hmeshi, \ u(\bfx)=0\ \forall \bfx\in D_{\edge},\ \edge \in \edgesexti \Bigr\} \subset \Hmeshi, i=1,\ldots,d.
\]
We then set $\Hmeshzero =  \prod_{i=1}^d  \Hmeshizero.$ 
Since we are dealing with piecewise constant functions, it is useful to introduce the characteristic functions $\characteristic_K$, for $K \in \mesh$, and $\characteristic_{D_\edge}$, for $\edge \in \edges$, defined by 
\[
 \characteristic_K(\bfx) = \begin{cases}
                 &1  \text{ if } \bfx \in K,\\ &0  \text{ if } \bfx \not \in K,\\
                \end{cases} 
\quad \characteristic_{D_\edge}(\bfx) = \begin{cases}
                  &1  \text{ if } \bfx \in D_\edge,\\ &0  \text{ if } \bfx \not \in D_\edge.
                \end{cases} 
\]

We can then write the functions $\bfu \in \Hmeshzero$ and $p, \vr \in L_\mesh$ as
\[
\bfu = (u_1,\ldots,u_d) \mbox{ with } u_i =  \displaystyle \sum_{\edge\in \edgesi_{\intt}} u_{\edge}\characteristic_{D_\edge},
\mbox{ for } i \in \llbracket 1, d \rrbracket,\quad p = \displaystyle \sum_{K \in \mesh} p_K\characteristic_{K}, \, \vr = \displaystyle \sum_{K \in \mesh} \vr_K\characteristic_{K}.
\]
\end{df}
%
% --------------------------------------------------------------------------------------------------------
%
\subsection{The numerical scheme}

Let  $\disc = (\mesh,\E) $ be a MAC grid of the computational domain $\Omega \subset \R^d$. Let  $h_\mesh$ be the size of the mesh.
Let $\alpha >1 $ and $\cstab >0$ be  given.
Let $ \bm{f} \in L^2(\Omega)^d $ and $M>0$, and let $ \vr^\star = M/|\Omega|$. 
We consider the following numerical scheme:

\textit{Find $(\bfu,p,\vr) \in  \Hmeshzero \times L_\mesh \times L_\mesh $ such that, a.e in $\Omega$},
\begin{subequations}\label{probdis}
\begin{align}\label{dcont} 
& \dv_\mesh^{\upw} (\vr \bfu) + \cstab h_\mesh^\alpha ( \vr- \vr^\star)= 0,
\\[2ex]
 & \dv_{\widetilde\E} (\vr \bfu \otimes \bfu) + \nabla_{\E} p - \mu \Delta_{\E} \bfu - (\mu+\lambda)\nabla_{\E} \dv_\mesh \bfu = \mathcal{P}_{\edges}\bm{f}, \label{dmom}
  \\[2ex]
&  p = \vr^\gamma, \ \vr \ge 0,
\label{deos}
\end{align}
\end{subequations}
where the discrete operators are defined hereafter for  each equation.%
% ----------------
%
\subsubsection{The mass balance equation}\label{massbalance}

Equation \eqref{dcont}  is a finite volume discretization of the mass balance ($\ref{cont2}$) over the primal mesh.  
The discrete function $\dv_\mesh^{\upw} (\vr \bfu)  \in L_\mesh  $ is defined by
\begin{equation*}
\dv_\mesh^{\upw} (\vr \bfu) (\bfx) = \frac{1}{|K|} \sum_{\edge \in \E(K)} F_{K,\edge}, \ \forall \bfx \in K,
\end{equation*}
where $F_{K,\edge}$
stands for the upwind mass flux across $\edge$ outward $K$, which reads:
\begin{equation}\label{eq:massflux}
\forall \edge \in\edges(K), \qquad F_{K,\edge}= |\edge| \ \vr_\edge^{\upw}\ u_{K,\edge} \; \mbox{ with } \vr_\edge^{\upw}=\left| \begin{aligned} &
\vr_K \qquad \mbox{if } u_{K,\edge} \geq 0,
\\[1ex] &
\vr_L \qquad \mbox{otherwise}, 
\end{aligned} \right.  
\end{equation}
and where $u_{K,\edge}$ is an approximation of the normal velocity to the face $\edge$ outward $K$, defined by:
\begin{equation}\label{eq:edge_velo}
	u_{K,\edge} =  u_{\edge} \ \bfe_i \cdot \bfn_{K,\edge}  \mbox{ for } \edge \in \edges\ei \cap \E(K),
\end{equation}
where $\bfn_{K,\edge} $ denotes the unit normal vector to $\edge$ outward $K$.
Thanks to the boundary conditions, $u_{K,\edge}$ vanishes for any external face $\edge$, and so does $F_{K,\edge}$.
Any solution $(\vr,\bfu) \in L_\mesh \times \Hmeshzero $ to ($\ref{dcont}$) satisfies $ \vr_K > 0$ for all $K \in \mesh$ so that in particular ($\ref{deos}$) makes sense:
the positivity of the density $\vr$ in ($\ref{dcont}$) is not enforced in the scheme but results from the above upwind choice.  Indeed, for any velocity field, the upwinding ensures that the discrete mass balance ($\ref{dcont}$) is a linear system for $\vr$ whose matrix is invertible and has a non negative inverse \cite[ Lemma C.3]{fettah2012numerical} and this gives $ \vr_K > 0$ for all $K \in \mesh$ (thanks to $\vr^\star>0$).

Note also that we have the usual finite volume property of local conservativity of the mass flux through a primal face  
$\edge=K|L$ (\ie \ $F_{K,\edge}=-F_{L,\edge}$). 
For $ \edge =\overrightarrow{K \vert L} \in \E_{\intt} $, we also define
\begin{equation}\label{gap}
[\vr]_\edge = \vr_L-\vr_K.
\end{equation}
The artificial term $\cstab h_\mesh^\alpha ( \vr- \vr^\star)$ guarantees that the integral of the density over the computational domain is always $M$.
Indeed, summing \eqref{dcont} over $K \in \mesh$, and using  the conservativity of the flux through a primal face, immediately yields the total conservation of mass, which reads:
\begin{equation}\label{masscons}
	 \int_\Omega \vr \dx = M.
\end{equation}
The constant $\cstab$ is chosen so that a uniform (with respect to the mesh) bound  holds on the solutions to ($\ref{probdis}$); these bounds are stated in Proposition $\ref{prop:estimates}$. The proof of this proposition shows that $\cstab$ can be chosen sufficiently small with respect to the data (see \eqref{cstestab}).
However, in practice, $C_s$ may be set to 1, in which case, the uniform bounds stated in Proposition \ref{prop:estimates}  hold for $h_\mesh $ sufficiently small.
 
%
% ----------------
%
\subsubsection{The momentum balance equation}

We now turn to the discrete momentum balances \eqref{dmom}, which are obtained by discretizing the momentum balance equation \eqref{mov2} on the dual cells associated to the faces of the mesh.
In the right hand side of \eqref{dmom},  $\mathcal P_\edges$ denotes the cell mean-value operator defined for $ \bv=(v_1,...,v_d) \in L^2(\Omega)^d $ by  
\begin{align}
& \mathcal{P}_{\edges}\bfv =\begin{pmatrix}\mathcal{P}_{\edges}^{(1)} v_1, \cdots, \mathcal{P}_{\edges}^{(d)} v_d \end{pmatrix} \in  H_{\edges,0}^{(1)}\times\cdots\times H_{\edges,0}^{(d)}, 
  \mbox{ where, for }i= 1, \ldots d, \nonumber \\[2ex]
& \begin{matrix}                                                                                         
     \mathcal{P}_{\edges}^{(i)}: & L^2(\Omega)\longrightarrow \Hmeshizero \hfill \\
			   & \displaystyle  \vi\;\longmapsto \mathcal{P}^{(i)}_{\edges}\vi =
			     \sum_{\edge \in \edgesint^{(i)}} \left( \frac{1}{\vert D_\edge\vert}\int_{D_\edge} \vi(\bfx) \dx \right) \characteristic_{D_\edge}. 
  \end{matrix}
\label{interpedges} \end{align}

\noindent {\bf The discrete convective operator -} The discrete divergence of $ \vr \bfu \otimes \bfu $  is defined by  
\begin{equation}
  \dv_{\widetilde\E} (\vr \bfu \otimes \bfu) = ( \dv^{(1)}_{\widetilde\E}(\vr  \bfu u_1),..., \dv^{(d)}_{\widetilde\E}(\vr  \bfu u_d)) \in \Hmeshzero,
\end{equation}
where the $i^{th}$ component of the above operator reads:
\[
 \dv^{(i)}_{\widetilde\E}(\vr  \bfu u_i) (\bfx) = \frac{1}{|D_\edge|} \sum_{\edged\in\edgesd(D_\edge)} F_{\edge,\edged}\ u_{\edged}, \ \forall \bfx \in D_\edge, \ \edge \in \E_{\intt}^{(i)}.
\]
The expression $\fluxd$ stands for the mass flux through the dual face $\edged$, and $u_\edged $ is an approximation of $i^{th}$ component of the velocity over $\edged$.

Let us consider the momentum balance equation for the $i^{th}$ component of the velocity, and $\edge \in \E^{(i)}_{\intt}$, $\edge=K|L$.
We have to distinguish two cases (see Figure \ref{figdualflux}):
\begin{list}{-}{\itemsep=0.ex \topsep=0.5ex \leftmargin=1.cm \labelwidth=0.7cm \labelsep=0.3cm \itemindent=0.cm}
\item First case -- The vector $\bfe_i$ is normal to $\edged$, in which case $\edged$ is included in a primal cell $K$; we then denote by $\edge'$ the second face of $K$ which is also normal to $\bfe_i$.
We thus have $\edged=D_\edge | D_{\edge'}$.
Then the mass flux through $\edged$ is given by:
\begin{equation}\label{eq:flux_eK}
F_{\edge,\edged} = \frac 1 2 \ \bigl[ F_{K,\edge}\ \bfn_{K,\edge} 
+ F_{K,\edge'}\ \bfn_{K,\edge'}  \bigr] \cdot \bfn_{D_\edge,\edged}.
\end{equation}
where  $\bn_{D_\edge,\edged} $  stands for the unit normal vector to $\edged$ outward $D_\edge$. 
\item  Second case -- The vector $\bfe_{i}$ is tangent to $\edged$, and $\edged$ is the union of the halves of two
primal faces $\edgeperp$ and $\edgeperp'$ such that $\edgeperp\in \edges(K)$ and $\edgeperp' \in \edges(L)$.
The mass flux through $\edged$ is then given by:
\begin{equation}\label{eq:flux_eorth}
F_{\edge,\edged} = \frac 1 2\ \bigl[F_{K,\edgeperp}+ F_{L,\edgeperp'} \bigr].
\end{equation}
\end{list}
\begin{figure}[htb]\label{figdualflux}
\centering
\begin{tikzpicture}[scale=1]
\draw[-](0,0)--(6,0)--(6,2)--(0,2)--(0,0);
\draw[-] (0,0)--(2,0)--(2,2)--(0,2)--(0,0);
\draw[fill=orange!10] (2,0)--(5,0)--(5,2)--(2,2)--(2,0);
\draw[-] (5,0)--(6,0)--(6,2)--(5,2)--(5,0);
\draw[-](4,0)--(4,2);
\path (0.5,1) node[] { $K$};
\path (5.75,1) node[] { $L$};
\path (3.8,0.75) node[rotate=90] { $\edge=K|L$};
 \path (4.5,0.14) node[] { $D_\edge$};
 \path (3.88,2.15) node[] { $\edged$};
\path (1.8,1) node[rotate=90] { $\edged \subset K$};
\path (1.6,2.2) node[] { $\edgeperp$};
\path (5.4,2.2) node[] { $\edgeperp'$};
 \draw[dashed,-,very thick,red](2,0)--(2,2);
 \draw[dashed,-,very thick,red](2,2)--(5,2);
 \draw[-,very thin,blue](0,2)--(4,2);
 \draw[-,very thin,green](4,2)--(6,2);
\end{tikzpicture}
\caption{Notations for the dual fluxes of the first component of the velocity.}
\end{figure}
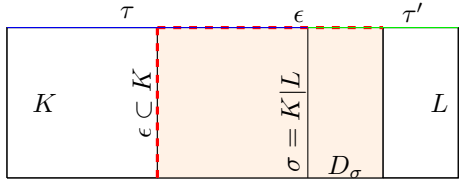
Note that we have the usual finite volume property of local conservativity of the mass flux through a dual face  
$D_\edge|D_{\edge'}$ (\ie \ $F_{\edge,\edged}=-F_{\edge',\edged}$), and that the flux through a dual face included in the boundary still vanishes.

The density on a dual cell is given by:
\begin{equation}
\begin{aligned} &
\mbox{for } \edge \in \edgesint,\ \edge=K|L\quad 
&&
|D_\edge|\ \vr_\Ds= |D_{ K,\edge}|\ \vr_K + |D_{ L,\edge}|\ \vr_L,
\\ &
\mbox{for } \edge \in \edgesext,\ \edge \in \edges(K),\quad
&&
\vr_\Ds= \vr_K.
\end{aligned}
\label{eq:rho_edge}\end{equation}
These definitions of the dual mass fluxes and the dual densities ensure that a finite volume discretization of the mass balance equation over the diamond cells holds:
\begin{equation}\label{eq:mass_D_imp}
\mbox{for } 1 \le i \le d, \ \forall \edge\in\edges_{\intt}^{(i)}, \qquad
\frac{1}{|D_\edge|} \sum_{\edged\in\edgesd(D_\edge)} \fluxd + \cstab h_\mesh^\alpha( \vr_{\Ds} - \vr^\star) 
 =0.
\end{equation}
This condition is essential to derive a discrete kinetic energy balance  in Proposition \ref{estimup} below.

Since the flux across a dual face lying on the boundary is zero, the values $u_\edged$ are only needed at the internal dual faces; they are chosen  centered   \ie,  \begin{equation*} \text{for} \  \edged=D_\edge|D_{\edge'} \in \widetilde{\E}_{\intt}^{(i)}, \quad u_{\edged} =\frac{u_{\edge} +u_{\edge'}}{2}.
\end{equation*}
\noindent {\bf Discrete divergence and gradient -}
 The discrete divergence operator $\dive_{\mesh}$ is defined by:
\begin{align} \label{eq:div}
  &\begin{array}{l| l} \displaystyle
    \dive_\mesh: \quad & \quad \Hmesh\longrightarrow L_{\mesh}
		\\ [1ex] & \displaystyle \quad \bfu \longmapsto \dive_\mesh \bfu  =  \sum_{K\in\mesh} \frac 1 {|K|} \sum_{\edge\in\edges(K)} \! \medge \uKedge \ \characteristic_K,
  \end{array}
\end{align}
where $u_{K,\edge} $ is defined in ($\ref{eq:edge_velo}$). 
Once again, we have the usual finite volume property of local conservativity of the flux through an interface $\edge=K|L$ between the cells  $K,L\in\mesh$, \ie $  u_{K,\edge} =-u_{L,\edge} ,\quad\forall \edge=K|L\in\edgesint.$
The discrete divergence of $\bfu = (u_1, \ldots, u_d) \in \Hmeshzero$   may also be written as
\begin{equation}
    \dive_\mesh \bfu  = \sum_{i=1}^d \sth (\eth_i u_i)_K \characteristic_K,
\end{equation}
where the discrete derivative $(\eth_i u_i)_K$ of $u_i$ on $K$ is defined by 
\begin{equation}
   (\eth_i u_i)_K = \frac{\vert \edge \vert} {\vert K \vert}(u_{\edge'} - u_{\edge})  \mbox{ with } K = [\overrightarrow{\edge \edge'}], \edge, \edge' \in \edgesi.   \label{discrete-derivative-i-ui}
\end{equation}
The pressure gradient in the discrete momentum balance  is defined as follows:
\begin{equation}\label{eq:grad}
\begin{array}{l|l}
\gradedges:\quad
& \quad
L_{\mesh} \longrightarrow \Hmeshzero 
\\[1ex] & \displaystyle \quad
p \longmapsto \gradedges p = (\eth_{1} p, \ldots,  \eth_{d} p)^t,
\end{array}
\end{equation}
where $\eth_i p \in  \Hmeshizero$ is the discrete derivative of $p$ in the $i^{th}$ direction, defined by: 
\begin{equation}
 \label{discderive}
 \eth_i p(\bfx) =  \frac{|\edge|}{|D_\edge|}\ (p_L - p_K)\, \quad \forall \bfx\in D_\edge, \ 
 \mbox{for } \edge=\overrightarrow{K|L} \in \edgesinti, \ i= 1,\ldots, d.
\end{equation}
Note that, in fact, the discrete gradient of a function of $L_\mesh$ should only be defined on the internal faces, and does not need to be defined on the external faces;
we set it here in  $\Hmeshzero$  (that is zero on the external faces) in order to be coherent with \eqref{dmom}.
This gradient is built as the dual operator of the discrete divergence, which means:

\begin{lm}[Discrete $\dive-\nablai$ duality]
\label{lem:duality}
Let $\Omega$ be a MAC-compatible bounded domain of $\R^d$, $d=2$ or $d=3$.
Let $\ q\in L_{\mesh}$ and $\bfv\in\Hmeshzero$.
Then we have:
\begin{equation}
\int_{\Omega} q  \ \dive_{\mesh}\bfv \dx +\int_{\Omega} \nabla_{\edges} q\cdot \bfv \dx  =0 \label{Ndiscret}.
\end{equation}
\end{lm}
\noindent {\bf Discrete Laplace operator} - For $i=1 \ldots, d$, we classically define the discrete Laplace operator on the $i^{th}$ velocity grid by:
\begin{align}
  &
\begin{array}{l|l}
-\Delta_{\edges}^{(i)} : \quad
& \quad
\Hmeshizero  \longrightarrow  \Hmeshizero
\\ & \displaystyle \quad
\ui \longmapsto - \Delta_{\edges}^{(i)}\ui  
\end{array} \nonumber \\
& 
- \Delta_{\edges}^{(i)}\ui (\bfx)=\frac{ 1}{ \vert D_{\edge} \vert}\sum_{\edged\in\tilde{\edges}(D_{\edge})} 
\phi_{\edge,\edged},\qquad \forall \bfx\in D_\edge,\ 
\mbox{ for } \edge \in \edgesi_{\intt}, \label{eq:lapi}
\end{align}
where
\begin{align}
&\displaystyle
\phi_{\edge,\edged}=
\begin{cases}
\ \dfrac{\vert\edged\vert}{d_{\edged}} (u_{\edge}-u_{\edge'}) &\mbox{ if } \edged= \edge\! \vert\edge' 
 \in \edgesdinti,\\[1ex]
\ \dfrac{\vert\edged\vert}{d_{\edged}} u_{\edge} &\mbox{ if } \edged \in\edgesdexti\cap\tilde{\edges}(D_{\edge})
\end{cases}
\end{align}
with $d_{\edged}$ given by \eqref{depsilon}.
The fluxes $\phi_{\edge,\edged}$ satisfy the local conservativity property: 
\begin{equation}
\label{conservdiff}
\phi_{\edge,\edged}=-\phi_{\edge',\edged},\quad\forall \edged=\edge\!\vert\edge' \in\edgesdinti.
\end{equation}
Then the discrete Laplace operator of the full velocity vector is defined by  
\begin{equation}
 \begin{array}{ll}
  -\Delta_\edges:  & \Hmeshzero \longrightarrow \Hmeshzero \\
  & \bfu \mapsto -\Delta_\edges\bfu = (-\Delta_{\edges}^{(1)} u_1,\ldots, -\Delta_{\edges}^{(d)} u_d)^t.
\end{array}
\end{equation}
Let us now recall the definition of the discrete $H^1_0$ inner product \cite{eymard2000finite}; it is obtained by taking the inner product of the discrete Laplace operator and a test function $\bfv\in\Hmeshzero$ and integrating over the computational domain.
A simple reordering of the sums (which may be seen as a discrete integration by parts) yields, thanks to the conservativity of the diffusion flux \eqref{conservdiff}:
\begin{equation}
  \label{ps}
  \begin{array}{l}
    \displaystyle \forall (\bfu, \bfv) \in \Hmeshzero^2, \qquad \int_\Omega -\Delta_\edges \bfu \cdot \bfv \dx= [\bfu,\bfv]_{1,\edges,0}=\sum_{i=1}^d [\ui,\vi]_{1,\edgesi,0}, \\
    [2ex]\mbox{with }[\ui,\vi]_{1,\edgesi,0} = \displaystyle\sum_{\substack{\edged \in \edgesdinti\\ \edged=\overrightarrow{\edge\!\vert\edge'}}} \frac{|\edged|}{d_\edged}\ (u_{\edge}-u_{\edge'})\ (v_{\edge}-v_{\edge'})+ \sum_{\substack{\edged \in \edgesdexti\\ \edged \in \edgesd(D_{\edge})}} \frac{|\edged|}{d_\edged}\ u_{\edge}\ v_{\edge}.
  \end{array}
\end{equation}
The bilinear forms $\left|\begin{array}{l}
                     \Hmeshizero \times \Hmeshizero \to \xR\\
                      [1ex] (u,v) \mapsto  [\ui,\vi]_{1,\edgesi,0}
                    \end{array}\right.
                   $
                   and
                   $\left|\begin{array}{l}
                     \Hmeshzero \times \Hmeshzero \to \xR\\
                      [1ex] (\bfu,\bfv) \mapsto  [\bfu,\bfv]_{1,\edges,0}
                    \end{array}\right.
                    $
are inner products on $\Hmeshizero$, for $i = 1, \ldots, d$, and on $\Hmeshzero$ respectively, which induce the following discrete $H^1_0$ norms:
\begin{subequations}
  \begin{align}
  &\label{normi}
  \|\ui\|^2_{1,\edgesi,0} = [\ui,\ui]_{1,\edgesi,0}= \sum_{\substack{\edged \in \edgesdinti\\ \edged=\overrightarrow{\edge\!\vert\edge'}}} \frac{|\edged|}{d_\edged}\ (u_{\edge}-u_{\edge'})^{2}+ \sum_{\substack{\edged \in \edgesdexti\\ \edged \in \edgesd(D_{\edge})}} \frac{|\edged|}{d_\edged}\ u_{\edge}^{2}\\
  & \label{normfull}
  \|\bfu\|^2_{1,\edges,0} = [\bfu,\bfu]_{1,\edges,0} = \sum_{i=1}^d \|\ui\|^2_{1,\edgesi,0}.
\end{align}
\label{norm}
\end{subequations}
\begin{figure}[tb]
\centering
\begin{tikzpicture}[scale=1]
\draw[-](0,0)--(6,0)--(6,2)--(0,2)--(0,0);
\draw[-] (0,0)--(2,0)--(2,2)--(0,2)--(0,0);
\draw[fill=orange!10] (2,0)--(5,0)--(5,2)--(2,2)--(2,0);
\draw[-] (5,0)--(6,0)--(6,2)--(5,2)--(5,0);
\draw[-] (0,2)--(4,2)--(4,3.5)--(0,3.5)--(0,2);
\draw[-] (4,2)--(6,2)--(6,3.5)--(4,3.5)--(4,2);
%\draw[-,very thick,blue=90!, pattern=north west lines, pattern color=blue!30] (2,1)--(5,1)--(5,2.75)--(2,2.75)--(2,1);
\draw[-,very thick,color=blue=90!] (2,1)--(5,1)--(5,2.75)--(2,2.75)--(2,1);
\draw[-](4,0)--(4,2);
\path (0.5,1) node[] { $K$};
\path (5.75,1) node[] { $L$};
\path (3.8,0.75) node[rotate=90] { $\edge=K|L$};
 \path (4.5,0.14) node[] { $D_\edge$};
 \path (2.5,2.54) node[] { $D_\edged$};
 \path (3.8,2.95) node[] { $\edge'$};
 \path (3.88,2.15) node[] { $\edged=\edge|\edge'$};
  \path (0.5,3.25) node[] { $M$};
 \path (5.75,3.25) node[] { $N$};
 \draw[-](2,2)--(2,3.5);
 \draw[-](5,2)--(5,3.5);
 \draw[-,very thick,red](2,2)--(5,2);
\end{tikzpicture}

\caption{Full grid for the definition of the derivative of the velocity.}
 \label{fig:gradient}
\end{figure}
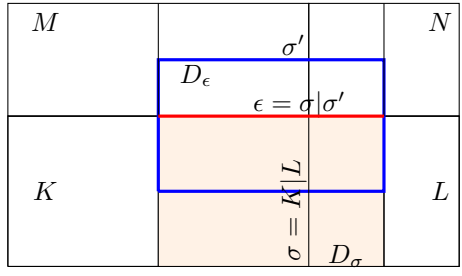
Since we are working on Cartesian grids, this inner product may be formulated as the $L^2$ inner product of discrete gradients. 
Indeed, we define the following discrete gradient of each velocity component $u_i$
\begin{equation}\label{partialdiscrete}
  \nabla_{\widetilde\edges^{(i)}} u_i = (\eth_1 u_i, \ldots, \eth_d u_i)  \mbox{ with }
  \eth_j u_i = \sum_{\substack{\edged \in \edgesdinti \\ \edged \perp \bfe_j}}  (\eth_j u_i)_{ D_\edged} \ \characteristic_{D_\edged}
  +\sum_{\edged \in \edgesdexti} (\eth_j u_i)_{ D_\edged} \ \characteristic_{D_\edged},
\end{equation}
where  $(\eth_j u_i)_{D_\edged} = \dfrac{u_{\edge'} - u_{\edge}}{d_\edged}$ with $\edged = \overrightarrow{\edge\vert \edge'}$, and  $D_\edged = \edged \times \bfx_\edge \bfx_{\edge'}$ (see Figure \ref{fig:gradient}, note also that $u_\edge=0$ if $\edge \in \edgesexti$).
This definition is compatible with the definition of the discrete derivative  $(\eth_i u_i)_K$ given by \eqref{discrete-derivative-i-ui},  since, if $\edged \subset K$, then $D_\edged = K$.
If   $\edged \in\edgesdexti\cap\tilde{\edges}(D_{\edge})$, we set
$(\eth_j u_i)_{D_\edged} = \dfrac{- u_{\edge}}{d_\edged}\bfn_{D_\edge,\edged}\cdot e_j$ with  $D_\edged = \edged \times \bfx_\edge \bfx_{\edge,b}$, where $\bfx_{\edge,b}=\edge \cap \partial \Omega$.
With this definition, it is easily seen that 
\begin{equation}\label{gradient-and-innerproduct}
    \int_\Omega \nabla_{\widetilde\edges^{(i)}} u \cdot \nabla_{\widetilde\edges^{(i)}} v  \dx = [u,v]_{1,\edgesi,0},
    \quad \forall u, v \in \Hmeshizero, \mbox{ for } i= 1, \ldots, d.
\end{equation}
where $[u,v]_{1,\edgesi,0}$ is the discrete $H^1_0$ inner product defined by \eqref{ps}.
We may then define 
$\nabla_{\widetilde\edges} \bfu= (\nabla_{\widetilde\edges^{(1)}} u_1, \ldots, \nabla_{\widetilde\edges^{(d)}} u_d),$
so that 
$
\displaystyle   \int_\Omega \nabla_{\widetilde\edges} \bfu : \nabla_{\widetilde\edges} \bfv  \dx = [\bfu,\bfv]_{1,\edges,0}.
$
An equivalent formulation of the discrete momentum balance \refe{dmom} reads:
\begin{multline}\label{eq:weakmom}
\int_\Omega \dv_{\widetilde\E}(\vr \bfu \otimes \bfu) \cdot \bfv \dx +
 \mu\int_\Omega \nabla_{\widetilde\edges} \bfu : \nabla_{\widetilde\edges} \bfv  \dx +(\mu+\lambda)\int_\Omega   \dive_{\mesh}\bfu \dive_{\mesh}\bfv \dx \\  -\int_\Omega p\, \dive_{\mesh}\bfv\dx = 
 \int_\Omega \mathcal P_\edges  \bff \cdot \bfv \dx, \ \forall \bfv \in \Hmeshzero.
\end{multline}
%
% ---------------------------------------------------------------------------------------------------------------------------------------------------
%
\section{Some analysis results for discrete functions}\label{disrot}

In the theory developed in this paper, we will need discrete Sobolev inequalites for the discrete approximations.
The following result is proved in \cite[Lemma 9.5 ]{eymard2000finite}.

\begin{thm}[Discrete Sobolev inequalities]\label{sobolev}
Let $\Omega$ be a MAC compatible bounded domain of $\R^d$, $d=2$ or $d=3$. Let $q< +\infty $ if $d=2$ and $ q=6$ if $d=3$. Then there exists $C = C(q,\Omega,\eta_\mesh)$, non increasing with respect to $\eta_\mesh$, such that, for all $ \bfu \in \Hmeshzero$,
\begin{equation*}
\| \bfu \|_{L^q(\Omega)} \le C  \| \bfu \|_{1,\E,0}. 
\end{equation*}
\end{thm}
The following compactness theorem is a consequence  of \cite[Theorem 9.1 and Lemma 9.5]{eymard2000finite} and \cite[Lemma 5.7]{eymard2009discretization}.
\begin{thm}
Let $\Omega$ be a MAC compatible bounded domain of $\R^d$, $d=2$ or $d=3$.
Consider a sequence of MAC grids $(\mesh_n,\E_n)_{n\in \xN}$, with step size $h_{\mesh_n}$ tending to zero as $\nti$.
Let $(\bfu_n)_\nnn$  be a sequence of discrete functions such that each element of the sequence belongs to   $\Hmeshnzero$  and such that the sequence $ (\| \bfu_n \|_{1,\E_n,0})_{n \in \xN}$  is bounded.
Then, up to the extraction of a subsequence, the sequence $(\bfu_n)_\nnn$  converges in $L^2(\Omega)^d$ to a limit $\bfu$  and this limit satisfies $\bfu \in (\xHone_0(\Omega))^d$.
Furthermore, one has $\nabla_{\widetilde\E_n} \bfu_n \tends \nabla \bfu$ weakly in $L^2(\Omega)^{d \times d}$ as $\nti$.
If $\eta_{\mesh_n} \ge \eta >0$, one has also $\bfu_n \to \bfu$ in $L^q(\Omega)$ for all $q < q(d)$.
\label{th:compactness}\end{thm}
We now recall a discrete analogue of the identity $\eqref{graddivrot}$ linking the gradient, divergence and curl operators, which is proved in \cite{eymard2010convergence}.
First of all, we modify the definition of the discrete gradient ($\nabla_{\E} $) of an element of $L_\mesh $ in some dual cells near the boundary, in order to take into account a null boundary condition at the external faces.
It reads:
\begin{equation}\label{eq:gradext}
\begin{array}{l|l}
\gradtg_{\E}:\quad
& \quad
L_{\mesh} \longrightarrow \Hmesh
\\[1ex] & \displaystyle \quad
w \longmapsto \gradtg_{\E} w = (\partdtg_{1} w , \ldots,  \partdtg_{d} w )^t,
\end{array}
\end{equation}
where $\partdtg_i w \in  \Hmeshi$ is the discrete derivative of $w$ in the $i^{th}$ direction, defined, for $i= 1,\ldots, d$, by: 
\begin{align}
 \label{discderiveext}
 &\displaystyle
\partdtg_i w(\bfx) = \begin{cases} \displaystyle
 \eth_i w(\bfx) =\frac{|\edge|}{|D_\edge|}\ (w_L - w_K), \ & \forall \bfx\in D_\edge, \ 
 \mbox{for } \edge=\overrightarrow{K|L} \in \edgesinti, \\[2ex] \displaystyle
 -\frac{|\edge|}{|D_\edge|}\ w_K\bn_{\edge,K} \cdot \bm{e}_i, \ & \forall \bfx\in D_\edge, \ 
 \mbox{for } \edge  \in \E(K) \cap \edgesexti. 
\end{cases}
\end{align}
In order to define the discrete $\curl$ operator of a function $ \bfv =(v_1,...,v_d) \in \Hmesh $, we use the functions
$ (\eth_j u_i)_{1 \le i,j \le d} $ defined in ($\ref{partialdiscrete}$). This definition is the same for $\bv \in \Hmeshzero$ and $\bv \in \Hmesh$,
the only difference is that we may have $u_\edge \ne 0$ if $\edge \in \edgesexti$ and $\bv \in \Hmesh$.
Then, the discrete $\curl$ operator of a function $ \bfv =(v_1,...,v_d) \in \Hmesh $ is defined by
\begin{align}\label{discurl}
&\displaystyle
\curld \bfv =
\begin{cases}
 \eth_1 v_2 -\eth_2 v_1 &\mbox{ if } d=2,
 \\[1ex]
 \Big(\eth_2 v_3 - \eth_3 v_2, \eth_3 v_1 - \eth_1 v_3,\eth_1 v_2- \eth_2 v_1 \Big)  &\mbox{ if } d=3,
\end{cases}
\end{align}
%where the functions $ (\eth_j u_i)_{1 \le i,j \le d} $ are introduced in ($\ref{partialdiscrete}$).
%More precisely, in ($\ref{partialdiscrete}$), we defined  the quantities $ (\eth_j u_i)_{1 \le i,j \le d} $ for $\bv \in \Hmeshzero$;
%they are naturally extended here to the case $\bv \in \Hmesh $.
The following algebraic identity is   a discrete version of ($\ref{graddivrot}$), which is exact in the case of the MAC scheme, contrary to the case of the non conforming P1 finite element scheme, see \cite{eymard2010convergent}.
\begin{lm} Let $\Omega$ be a MAC compatbile bounded domain of $\R^d$, $d=2$ or $d=3$ and let $\mesh$ be a MAC grid and $(\bfv,\bfw) \in(\Hmeshzero)^2$.
Then the following discrete identity holds:
\begin{equation}
\int_\Omega \nabla_{\widetilde\E} \bfv : \nabla_{\widetilde\E} \bfw \dx =
\int_\Omega \dived \, \bfv \ \dived \bfw \dx  + \int_\Omega \curld  \bfv \,\cdot \curld \bfw \dx.
\label{gdcd}
\end{equation}
\label{ggddccd}
\end{lm}

We finish this section by introducing a discrete construction of the test function used in Step 3 of the proof of Theorem \ref{continuityws} to obtain the convergence of the so-called effective viscous flux.
We recall that this test function is the product of a scalar regular function with a velocity field whose divergence is the density; we need here to show the existence, at the discrete level, of such a velocity field, and then some regularity estimates for the resulting test function.
To this goal, we first introduce the discrete Laplace operator on the primal mesh.
For $\edge \in \edgesint$, $\edge=K|L$, let $d_\edge$ be defined as the distance between the mass center of $K$ and $L$, \ie \ $d_\edge=d(\bfx_K,\bfx_L)$; for an external face $\edge \in \edgesext$ adjacent to the primal cell $K$, let $d_\edge=d(\bfx_K,\edge)$.
Then, with this notation, we obtain a discretization of the Laplace operator wih homogeneous Dirichet boundary conditions on the primal mesh by:
\begin{align}
  &
\begin{array}{l|l}
-\Delta_{\mesh} : \quad
& \quad
L_\mesh  \longrightarrow  L_\mesh
\\ & \displaystyle \quad
w \longmapsto - \Delta_{\mesh} w  
\end{array} \nonumber \\
& 
- \Delta_{\mesh} w (\bfx)=\frac{ 1}{ \vert K \vert}\sum_{\edge \in {\edges}(K)} 
\phi_{K,\edge},\qquad \forall \bfx\in K,\ 
\mbox{ for } K \in \mesh, \label{eq:lap2}
\end{align}
where
\begin{align}
&\displaystyle
\phi_{K,\edge}=
\begin{cases}
\ \dfrac{\vert \edge\vert}{d_\edge} (w_{K}-w_{L}) &\mbox{ if } \edge= K \! \vert L
 \in \edgesint,\\[2ex]
\ \dfrac{\vert\edge\vert}{d_\edge} w_{K} &\mbox{ if } \edge \in\edgesext\cap {\edges}(K).
\end{cases}
\end{align}
The following lemma \cite{eymard2010convergence} clarifies the relations between this Laplace operator and the already defined gradient  divergence and curl operators.

\begin{lm} Let $\Omega$ be a MAC compatible bounded domain of $\R^d$, $d=2$ or $d=3$.
Let $w \in L_{\mesh}$.
Let $\bfv = - \gradtg_{\E} w \in \Hmesh$ be defined by~\eqref{eq:gradext}.
Then, with the discrete  $\curl$ operator  defined by \eqref{discurl}, we have $\curld{\bfv}=0$.
Furthermore, for any $\vr \in L_{\mesh}$, there exists one and only one $w$ in $L_{\mesh}$ such that $- \Delta_\mesh w = \vr$, and, in this case, $\dived{\bfv}=\vr.$
\label{divcurl}\end{lm}

Now, to any regular function $\varphi  \in \xC^\infty_c(\Omega)$, we associate an interpolant $\varphi_\mesh \in L_\mesh$ defined by:
\begin{equation}
\varphi_\mesh (\bfx) =\varphi(\bfx_K) \textrm{ for all } \bfx \in K, \ \forall K \in \mesh.
\label{apphi}\end{equation}
We are now in position to state the following discrete regularity result (see \cite{eymard2010convergence} for a proof).

\begin{lm}Let $\Omega$ be a MAC compatible bounded domain of $\R^d$, $d=2$ or $d=3$. Let $\disc=(\mesh,\E)$ be a MAC grid. Let $ \vr \in L_\mesh $ and $ w \in L_\mesh $ be defined by
\begin{equation}\label{fvle}
-\Delta_\mesh w = \vr.
\end{equation} 
Let $\varphi \in \xC^\infty_c(\Omega)$ and $\gradtg (w\varphi_\mesh)$ be the gradient of the function $w\varphi_\mesh$ as defined in \eqref{eq:gradext}.
Then there exists $C_\varphi$ only depending on $\varphi$, $\Omega$ and on $\eta_\mesh$ in a non increasing way such that $ \| \gradtg_{\E}(w \varphi_\mesh)) \|_{1,\E,0} \le  C_\varphi \norm{\vr}{L^2(\Omega)}$, where $ \| \cdot \|_{1,\E,0} $ is defined in ($\ref{normfull}$).
\label{locest}\end{lm}
%
% -------------------------------------------------------------------------------------------------------------------------------------------
%
\section{Main theorem}

Now, we are ready to state the main result of this paper.
We recall the notation: 
\begin{equation*} q(d)= \left\{
\begin{array}{l}
  +\infty \ \text{if} \ d=2, \\
 6 \ \text{if} \ d=3.
\end{array}
\right.
\end{equation*}

\begin{thm}\label{mainthm}
Let $\Omega$ be a MAC compatible bounded domain of $\R^d$, $d=2$ or $d=3$. 
Let $\bm{f} \in (L^2(\Omega))^d, M>0,$ and $ \alpha > 1$. Let $\gamma > 3 $ if $d=3$ and $ \gamma > 1 $ if $d=2$. 
Consider a sequence of MAC grids $(\disc_n=(\mesh_n,\E_n))_{n\in \xN}$, with step size $h_{\mesh_n}$ going to zero as $\nti$. 
Assume that there exists $\eta > 0 $ such that  $ \eta \le \eta_{\mesh_n} $ for all $ n \in \mathbb{N} $, where $ \eta_{\mesh_n} $ is defined by ($\ref{regmesh}$).
For a value of the constant $\cstab$ independent of $\nnn$ and sufficiently small with respect to the data, there exists a  solution $(\bu_n,p_n,\vr_n) \in  \Hmeshnzero \times L_{\mesh_n}(\Omega) \times L_{\mesh_n} (\Omega) $ to the scheme \eqref{probdis} with any of the MAC discretizations $\disc_n$; in addition, the obtained density and pressure are positive a.e. in $\Omega$.
Furthermore, up to a subsequence:
\begin{itemize}
\item the sequence $(\bfu_n)_\nnn$ converges in $(L^q(\Omega))^d$ for any $q \in [1,q(d))$ to  a function $\bfu \in \xHone_0(\Omega)^d$, and $(\nabla_{\E_n} \bu_n)_{\nnn} $ converges weakly to $\nabla \bu $ in $L^2(\Omega)^{d \times d} $,
\item the sequence $(\vr_n)_\nnn$ converges in $L^{p}(\Omega)$ for any $p$ such that $1 \leq p < 2 \gamma$ and weakly in $L^{2\gamma}(\Omega)$ to a function $\vr$ of $L^{2\gamma}(\Omega)$,
\item the sequence $(p_n)_\nnn$ converges in $L^{p}(\Omega)$ for any $p$ such that $1 \leq p < 2$ and weakly in $L^2(\Omega)$ to a function $p$ of  $L^2(\Omega)$,
\item $(\bu,p,\vr)$ is a weak solution of Problem \eqref{pbcont_w}--\eqref{pressure1}   in the sense of Definition \ref{ldeuxw}.
\end{itemize}
\label{theo:conv1}
\end{thm}
The convergence part of Theorem \ref{mainthm} remains true with a fixed value of $C_s$ (for instance, $C_s=1$). 
The only difference is that the estimates on the
approximated solutions are valid only for $h_\mesh$ small enough with respect to the data.
The following sections are devoted to the proof of Theorem \ref{theo:conv1}. For the sake of clarity, we shall perform the proofs only in the three-dimensional case (and then $\gamma>3$).
The modifications to be done for the two-dimensional case, which is in fact simpler, are mostly due to the different Sobolev embeddings and are left to the interested reader. Throughout the proof of this theorem, we adapt to the discrete case the strategy followed to prove Theorem \ref{continuityws}.
%
% -------------------------------------------------------------------------------------------------------------------------------------------
%
\section{Mesh independent estimates}\label{4}

\subsection{Notations}

From now on, we assume that $\Omega$ is a MAC compatible bounded domain of $\R^d$, $d=2$ or $d=3$,   and that all the considered meshes satisfy $\eta \le \eta_\mesh $, for a given $\eta>0$ and with $\eta_\mesh $ defined by ($\ref{regmesh}$). 
The letter $C$ denotes positive real numbers that may tacitly depend on $|\Omega|$, ${\rm diam}(\Omega)$, $\gamma$, $\lambda$, $\mu$, $M$, $\bm{f}$, $\alpha$, $\eta$ and on other parameters; the dependency on these other parameters (if any) is always explicitly indicated.
These numbers can take different values, even in the same formula.
They are always independent of the size of the discretisation $h_\mesh$.
%
%----------------------------------------
%
\subsection{Existence}

Let us now state that the discrete problem ($\ref{probdis}$) admits at least one solution.
This existence result follows from a the topological degree argument (see \cite{dei-85-non} for the theory, \cite{eymard1998error} for the first application to a nonlinear numerical scheme and Appendix \ref{existproof} for the proof).
\begin{thm}\label{thmexist}
There exists a solution $(\bfu,p,\vr) \in \Hmeshzero \times L_\mesh \times L_\mesh$ to Problem ($\ref{probdis}$). Moreover any solution is such that $ \vr > 0 $ a.e in $\Omega$ (in the sense that $ \vr_K>0,\ \forall K \in \mesh$).
\end{thm}
%
%----------------------------------------
%
\subsection{Energy Inequality}

Let us now turn to stability issues: in order to prove the convergence of the scheme, we wish to obtain some uniform (with respect to the mesh) bounds on the solutions to \eqref{probdis}, see  Proposition \ref{prop:estimates} below.
We begin by a technical lemma \cite[Lemma 5.4]{eymard2010convergence} which is useful  not only for stability issues, but also for the three following reasons.
First, it allows an estimate on $\bfu$ in a dicrete $\xH^1_0$ norm (Proposition \ref{prop:estimates}), as in \cite[Proposition 5.5]{eymard2010convergence}. 
Second, it yields a so called weak BV estimate, which  depend on the mesh and does not give a direct compactness result on the sequence of approximate solutions; however it is useful
 in the passage to the limit in the mass equation, in the discrete convective term and in the equation of state.
Third, Lemma \ref{lmm:renorm}  gives (with $\beta=1$) a crucial inequality which is also used in order to pass to the limit in the equation of state. 
\begin{lm}
Let $\vr \in L_\mesh$ and $\bfu \in \Hmeshzero$ satisfy \eqref{dcont}.
Then, for any $\beta \geq 1$:
\[
\int_\Omega \vr^\beta \dived \bfu \dx
+\frac 1 2 \sum_{\edge \in \edgesint} \beta\, |\edge|\  \vr_{\edge,\beta} \ |\uedge|\ [\vr]_\edge^2
\leq C C_s\ h_\mesh^\alpha,
\]
where $C$ depends only on $\mass$, $\beta$, $\mu$, $\alpha$, $\Omega$ and $\eta$, and, for any $\edge \in \edgesint$, $\edge=K|L$,
\begin{equation}  
\vr_{\edge,\beta} = \min( \vr_K^{\beta-2}, \vr_L^{\beta-2}). 
\label{rhosigmabeta}
\end{equation}

\label{lmm:renorm}\end{lm}

In order to obtain an  estimate on the pressure, we need to introduce a so-called Fortin interpolation operator, \ie\ an operator which maps velocity functions to discrete functions and preserves the divergence. 
The following lemma is given in  \cite[Theorem 1]{gallouet2012w1}, and we repeat it here with our notations for the sake of clarity.
We will use this Lemma later on with $p=2$.

\begin{lm}[Fortin interpolation operator] \label{lem:fortin}
Let $\mathcal D= (\mesh, \edges)$ be a MAC grid of $\Omega$. Let $1 \le p <\infty$.
For $\bfv=(v_1,...,v_d) \in (W_{0}^{1,p}(\Omega))^{d}$ we define $\widetilde{\mathcal P}_{\edges} \bfv$  by 
\begin{align}
  &\widetilde{\mathcal P}_{\edges}\bfv =\begin{pmatrix}\widetilde{\mathcal P}_{\edges}^{(1)} v_1, \cdots, \widetilde{\mathcal P}_{\edges}^{(1)} v_d \end{pmatrix}\in  \Hmeshzero, 
  \mbox{ where for }i= 1, \ldots d, \nonumber\\
  & \begin{array}{ll}                                                                                         
    \widetilde{\mathcal P}_{\edges}^{(i)}: & W_{0}^{1,p}(\Omega) \longrightarrow \Hmeshizero \\
					 & \vi\;\longmapsto  \widetilde{\mathcal P}_{\edges}\vi \mbox{ defined by}\\
			   & \qquad \displaystyle  \widetilde{\mathcal P}^{(i)}_{\edges}\vi(\bfx)=(\widetilde{\mathcal P}^{(i)}_{\edges}\vi)_\edge=\frac{1}{\vert\edge\vert}\int_{\edge} \vi(\bfx) \dgammax, \  \forall \bfx\in D_{\edge}, \  \edge\in\edgesi.
\end{array} \label{interp-moyenne}
\end{align}
Then  $\widetilde{\mathcal P}_{\E}$ satisfies:
\begin{equation}\label{strongcvfortin}
\|  \widetilde{\mathcal P}_{\E} \bvarphi - \bvarphi \|_{L^\infty(\Omega)} \le C_{\bvarphi} h_\mesh, \ \forall \bvarphi \in C_c^\infty(\Omega)^d.
\end{equation}
For $q\in L^{1}(\Omega)$, we define $\mathcal{P}_{\mesh}q  \in L_\mesh$ by:
\begin{equation}
    \mathcal{P}_{\mesh} q (\bfx) = \frac 1 {\vert K \vert} \int_K q (\bfx) \dx. \label{Pmesh}
\end{equation}
Let $\eta_\mesh >0$ be defined by \eqref{regmesh}.
Then, for $\bvarphi\in (W_{0}^{1,p}(\Omega))^{d}$, 
\begin{subequations}
\begin{align}
  &  \dive_{\mesh}(\widetilde{\mathcal P}_{\edges}\bvarphi) ={\mathcal P}_{\mesh} (\dive  \bvarphi),
  \label{conserv-div-interp} \\
  & \| \nabla_{\widetilde\E} \widetilde {\mathcal P}_{\edges} \bvarphi\|_{(L^p(\Omega)^{d \times d})}\leq  C_{\eta_{\mesh}} \|\nabla\bvarphi\|_{(L^{p}(\Omega))^{d}}, \label{norme-h1-interp}
\end{align}
\end{subequations}
where  $C_{\eta_\mesh}$ depends only on $\Omega$, $p$ and on $\eta_\mesh$ in a decreasing way.
\end{lm}

We can now state and prove the estimates on a discrete solution that we are seeking.
These estimates may be seen as an equivalent for the discrete case of Step 1 of the proof of Theorem \ref{continuityws}.

\begin{prop}\label{prop:estimates}
Let $(\bfu,p,\vr) \in  \Hmeshzero \times L_\mesh \times L_\mesh$  be a solution to the scheme, \ie\ system \eqref{probdis}.
Taking $C_s$ small enough with respect to the data (namely $\mu$, $\mass$, $\Omega$, $\alpha$, $\eta$)
there exists  $C_1$  depending only on $\bff$, $\mu$, $\mass$, $\Omega$, $\gamma$, $\alpha$ and on $\eta$  such that:
\begin{equation}\label{ienergie}
\| \bfu \|_{1,\E,0} +  \norm{p}{L^2(\Omega)} + \norm{\vr}{L^{2\gamma}(\Omega)} \leq C_1.
\end{equation}\label{estimup}

Moreover, for any $\beta \in [1,\gamma],$ there exists $C_2$ depending only on $\bff$, $\mass$, $\Omega$, $\gamma$, $\mu$, $\alpha$, $\beta$ and $\eta$ such that 
\begin{equation}
\sum_{\edge \in \edgesint} |\edge| \ \vr_{\edge,\beta}\ |\uedge|\ [\vr]_\edge^2 \leq C_2,
\label{rho_jump}
\end{equation}
where $\vr_{\edge,\beta}$ is defined in \eqref{rhosigmabeta}.
In particular, since $ \gamma>3$, we get by taking $\beta = 2$ in \eqref{rho_jump}: 
\begin{equation}
\sum_{\edge \in \edgesint} |\edge|\  |\uedge|\ [\vr]_\edge^2 \leq C_2.
\label{bvweak}\end{equation}
\end{prop}

\begin{proof}
In order to prove Proposition \ref{prop:estimates}, we proceed in several steps. 
We follow the proof established in the continuous case to obtain uniform bounds of the approximate solutions.

\vspace{1cm}

\textbf{Step 1} : Estimates on $ \| \bfu \|_{1,\E,0} $ and inequality ($\ref{rho_jump}$).

Taking $\bfu$ as a test function in \eqref{eq:weakmom}, using the Hold\"er's inequality and thanks to the fact that the discrete $\xHone$ norm controls the $L^2$  norm (see Theorem \ref{sobolev}) , we have:
\begin{multline}
\frac{\mu}{2} \| \bfu \|_{1,\E,0}^2+(\mu+\lambda) \| \dived\bfu \|_{L^2(\Omega)}^2 -\int_\Omega p \, \dived\bfu  \dx 
\\ +\sum_{i=1}^3 \sddi \frac{1}{2}F_{\edge,\edged} (u_{\edge} +u_{\edge'})( u_{\edge}-u_{\edge'}) \le C
\end{multline}
where $C$ depends only on $\bff$  and  $\Omega$.
Moreover, by virtue of ($\ref{eq:mass_D_imp}$),
\begin{multline*}
\sum_{i=1}^3 \sddi \frac{1}{2}F_{\edge,\edged} (u_{\edge} +u_{\edge'})( u_{\edge}-u_{\edge'}) =\sum_{i=1}^3 \sddi \frac{1}{2}F_{\edge,\edged}( (u_{\edge})^2-(u_{\edge'})^2) \\
= \sum_{i=1}^3 \sum_{\edge \in \E_{\intt}^{(i)}} \frac{(u_{\edge})^2}{2} \sum_{ \edged \in \widetilde{\E}(D_\edge)}  F_{\edge,\edged}  = -\frac{1}{2} \cstab h_\mesh^\alpha ( \int_\Omega \vr \| \bfu \|^2 \dx - \vr^\star \int_\Omega \| \bfu \|^2 \dx)
\end{multline*}

Lemma \ref{lmm:renorm} with $\beta=\gamma$ yields, since $p=\vr^\gamma$:
\[
\int_\Omega p \, \dived \bfu  \dx +\frac 1 2 
\sum_{\edge \in \edgesint} \gamma\, |\edge|\ \vr_{\edge,\gamma}\ |\uedge|  [\vr]_\edge^2 \leq C,
\]
where $C$ depends only on $\mass$, $\gamma$, $\alpha$, $\mu$, $\Omega$ and $\eta$.

Consequently
\begin{equation*}
\frac{\mu}{2} \| \bfu \|_{1,\E,0}^2 +\frac 1 2 
\sum_{\edge \in \edgesint} \gamma\, |\edge|\ \vr_{\edge,\gamma}\ |\uedge|  [\vr]_\edge^2  \le \frac{1}{2}\cstab \mass h^\alpha \| | \bfu | \|_{L^\infty(\Omega)}^2 +C
\end{equation*}
By virtue of Theorem \ref{sobolev} we have $ h_\mesh^3 \| \bfu \|_{L^\infty (\Omega)^3}^6 \le C(\eta) \| \bu \|_{L^6(\Omega)^3}^6 \le C(\eta) \| \bfu \|_{1,\E,0}^6 $ and therefore 
\begin{equation*}
\|   \bfu  \|_{L^\infty (\Omega)^3} \le C(\eta) \frac{1}{\sqrt{h_\mesh}} \| \bfu \|_{1,\E,0}.
\end{equation*}
Summing these two relations, we thus obtain:
\begin{equation}
\frac{\mu}{2} \| \bfu \|_{1,\E,0}^2 +\frac 1 2 
\sum_{\edge \in \edgesint} \gamma\, |\edge|\ \vr_{\edge,\gamma}\ |\uedge|\ [\vr]_\edge^2 \leq C 
+ \frac{1}{2} C(\eta)\cstab M  h_\mesh^{\alpha-1} \| \bfu \|_{1,\E,0}^2
\label{stab1}\end{equation}
and consequently, since $\alpha > 1 $,
\begin{equation*}
\frac{1}{2} ( \mu - C(\eta) \cstab M \diam(\Omega)^{\alpha-1} ) \| \bfu \|_{1,\E,0}^2+\frac 1 2 
\sum_{\edge \in \edgesint} \gamma\, |\edge|\ \vr_{\edge,\gamma}\ |\uedge|\ [\vr]_\edge^2   \leq C.
\end{equation*}
%debut rh ref note M4
Let us choose $ \cstab $ such that $ 0 < \cstab < \frac{\mu}{C(\eta) M \diam(\Omega)^{\alpha-1}} $;  a possible choice is:
\begin{equation}\label{cstestab}
0< \cstab < \frac{\mu \eta^6}{  M \diam(\Omega)^{\alpha-1}}.
\end{equation}
Then
\begin{equation*}
\| \bfu \|_{1,\E,0} +\frac 1 2 
\sum_{\edge \in \edgesint} \gamma\, |\edge|\ \vr_{\edge,\gamma}\ |\uedge|\ [\vr]_\edge^2  \leq C.
\end{equation*}
\textbf{Step 2}: Estimate on $ \| p \|_{L^2(\Omega)} $.

Let  $m(p)$ stand for the mean value of $p$. By Lemma \ref{lem:bogos}, there exists  $ \bv=(v_1,v_2,v_3) \in H_0^1(\Omega)^3$ such that 
\begin{equation*}
\left\{
\begin{array}{l}
 \dive \bv=p-m(p), \\
  \| \bv \|_{H_0^1(\Omega)^3} \le C(\Omega) \| p-m(p) \|_{L^2(\Omega)},
\end{array}
\right.
\end{equation*}
Multiplying \eqref{dmom} by $ \widetilde{\mathcal P}_{\edges}\bfv $ (where $ \widetilde{\mathcal P}_{\edges} $ is defined in Lemma \ref{lem:fortin}) and integrating over $\Omega$ we have:
\begin{equation*}
  \| p-m(p) \|_{L^2(\Omega)}^2 \le C \| p-m(p) \|_{L^2(\Omega)}+\sum_{i=1}^3 \sddi F_{\edge,\edged} \frac{1}{2}(u_{\edge}+u_{\edge'}) (  (\widetilde{\mathcal P}_{\edges}^{(i)}v_{i})_\edge - (\widetilde{\mathcal P}_{\edges}^{(i)}v_{i})_{\edge'})
\end{equation*}
where C depends on $ \bm{f}, \Omega, \eta,\mu, \alpha, \gamma,M $. Now keeping in mind the definition of the dual fluxes (see ($\ref{eq:flux_eK}$) and ($\ref{eq:flux_eorth}$)) and the definition of $ \| \cdot \|_{1,\E,0} $, a technical but straightforward computation gives
\begin{align*}
\Big| \sum_{i=1}^3 \sddi F_{\edge,\edged} \frac{1}{2}(u_{\edge}+u_{\edge'}) (  (\widetilde{\mathcal P}_{\edges}^{(i)}v_{i})_\edge - (\widetilde{\mathcal P}_{\edges}^{(i)}v_{i})_{\edge'})\Big|  &\le C \| \vr \|_{L^6(\Omega)} \| \bfu \|_{L^6(\Omega)}^2  \|\widetilde{\mathcal P}_{\edges}\bfv \|_{1,\E,0} 
\\ &\le C \| p \|_{L^2(\Omega)}^{\frac{1}{\gamma}} \| p -m(p) \|_{L^2(\Omega)},
\end{align*}
where C depends on $ \bm{f}, \Omega, \eta,\mu, \alpha, \gamma,M $. 
The last inequality is obtained thanks to the the energy inequality ($\ref{ienergie}$) to get a bound on $ \| \bfu \|_{L^6 (\Omega)} $ (thanks to Theorem $\ref{sobolev}$) and  H\"older's inequality  since $ 2\gamma \ge 6 $ and $ p = \vr^\gamma$.
Consequently
\begin{equation*}
\| p-m(p) \|_{L^2(\Omega)} \le C( \| p \|_{L^2(\Omega)}^{\frac{1}{\gamma}} +1)
\end{equation*}
where $C$ depends on $\bff$, $\mu$, $\mass$, $\Omega$, $\gamma$, $\alpha$ and on $\eta$. Since $\int_\Omega p^{\frac{1}{\gamma}} \dx = \int_\Omega \vr \dx =M$, Lemma \ref{estl2} gives an $L^2$ bound for $p$ depending only on the data. To conclude, we obtain a $L^{2\gamma}$ bound for the density since $p = \vr^\gamma$.

In order to prove ($\ref{rho_jump}$) for $1 \le \beta \le \gamma$, let us use once again Lemma \ref{lmm:renorm}, to obtain:
\[
\frac 1 2 \sum_{\edge \in \edgesint} \beta\, |\edge|\  \vr_{\edge,\beta} \ |\uedge|\ [\vr]_\edge^2
\leq -\int_\Omega \vr^\beta \dived \bfu \dx+ C,
\]
where $C$ depends on $ M,\beta,\mu,\alpha,\Omega,\eta$. Since $\vr$ is bounded in $L^{2\beta}(\Omega)$ and $ \| \dv_\mesh \bfu \|_{L^2(\Omega)} $ is controlled by $\| \bfu \|_{1,\E,0} $, this
concludes the proof.
\end{proof}
Note that if, in Proposition \ref{prop:estimates}, we choose a fixed value of $C_s$, for instance $C_s=1$, There exists $\bar h>0$, depending of the data, such that the conclusions of Proposition \ref{prop:estimates} are true for  
$h_\mesh \le \bar h$. This is easy to see with \refe{stab1}.

%
% -------------------------------------------------------------------------------------------------------------------------------------------
%

\section{Convergence analysis}\label{sec:qdm_and_mass}

The aim of this section is to pass to the limit in the discrete equations ($\ref{dcont}$)--($\ref{deos}$). As in the continuous case, thanks to the estimates established in the previous section, taking a sequence of meshes, we can assume the convergence, up to a subsequence, of the  discrete solution to some $(\bu,p,\vr)$, in a convenient sense. We will first prove that $(\bfu,p,\vr)$ satisfies the weak form of Problem \refe{pbcont_w}-\refe{ci2}. We then prove that $ p=\vr^\gamma $.
The first difficulty is the convergence of the discrete convective term (the second  consists in passing to the limit in the equation of state). Indeed it is not easy to manipulate the discrete convective operator defined with the dual fluxes.
We then introduce velocity interpolators  in order to  transform the discrete convective operator. 
It relies on the reconstruction of each velocity component on all faces (or edges in 2D) of the mesh.
Similar results are used in \cite{gallouet2015convergence} for the  incompressible case.
%
% -------------------------------
%
\subsection{Passing to the limit in the mass and momentum balance equations}

\begin{lm}[Velocity interpolators]\label{lem:uchapeau}
For a given MAC grid $\disc =(\mesh, \edges)$, we define, for $i,j= 1,2,3$, the  full grid velocity reconstruction operator with respect to $(i,j)$ by 
 \begin{align} \mathcal R_\edges^{(i,j)} :\; & \; \Hmeshizero \to H_{\E,0}^{(j)} \nonumber\\ 
					  & v \mapsto \mathcal R_\edges^{(i,j)}  v =  \sum_{\edge\in \edgesintj} (\mathcal R_\edges^{(i,j)}  v)_\edge \characteristic_{D_\edge}, \label{def:ufull}
\end{align}
where 
%  \begin{align}
%    &\label{uhat}
%    (R_\edges^{(i,j)}  v)_\edge = v_\edge \mbox{ if } \edge \in \edges_{\intt}^{(i)},
%    \quad (R_\edges^{(i,j)}  v)_\edge = \frac{1}{4} \sum_{\edge'\in\mathcal{N}_\edge} v_{\edge'} \mbox{ otherwise}, \\
%    &\mbox{where, for any } \edge=K|L \in\edgesint \setminus\edgesi_{\intt},\  \mathcal{N}_\edge=\{\edge'\in\edgesi, \edge' \in \edgesK \cup \edgesL\}. \label{nedge} 
%\end{align}
 \begin{align}
    &\label{uhat}
    (\mathcal R_\edges^{(i,i)}  v)_\edge = v_\edge \textrm{ for } \edge \in \edges_{\intt}^{(i)},
    \\
    &\nonumber \textrm{and, for }\edge=K|L \in\edgesintj, \, j \ne i,
    \\
    &(\mathcal R_\edges^{(i,j)}  v)_\edge = \frac{1}{4} \sum_{\edge'\in\mathcal{N}_\edge} v_{\edge'}, 
   \, \,  \mathcal{N}_\edge=\{\edge'\in\edgesi, \edge' \in \edgesK \cup \edgesL\}. \label{nedge} 
\end{align}

For any $i=1,2,3,$ we also define a projector from $ \Hmeshi $ into $L_\mesh$ by
 \begin{align} \mathcal R_\mesh^{(i)} :\; & \; \Hmeshi \to L_\mesh \nonumber\\ 
					  & v \mapsto \mathcal R_\mesh^{(i)}  v =  \sum_{K \in \mesh} (\mathcal R_\mesh^{(i)}  v)_K \ \characteristic_{K}, \label{def:ufullprimal}
\end{align}
where 
  \begin{equation}
    \label{eq:interpolate_primal}
     (\mathcal R_\mesh^{(i)}  v)_K = \frac{1}{2} \sum_{\edge \in \edges^{(i)}(K)} v_\edge.
\end{equation}
Then there exists $C \ge 0$, depending only on the regularity of the mesh (defined by ($\ref{regmesh}$)) in a decreasing way, such that, for any $1 \le q < \infty$ and for any $i,j = 1,2, 3$, 
\begin{equation*} 
\Vert \mathcal R_\edges^{(i,j)} v \Vert_{L^q(\Omega)}  \le C \Vert v \Vert_{L^q(\Omega)} \textrm{ for any } v \in \Hmeshizero,
\end{equation*}
\begin{equation*} 
\Vert \mathcal R_\mesh^{(i)}  v \Vert_{L^q(\Omega)}  \le C \Vert v \Vert_{L^q(\Omega)} \textrm{ for any } v \in \Hmeshi.
\end{equation*}

\begin{proof}
 Let us prove the bound on $\Vert \mathcal R_\edges^{(i,j)} v\Vert_{L^q(\Omega)}$ for $d= 2$,  $i=1$ and $j=2$. 
 The other cases are similar. 
 In this case, for a given $\edge=K|L \in \edgesinti$, the edge $\edge$ belongs to  $ \mathcal{N}_{\edge'}$
 for  $\edge' \in \{{\edge_K^t}, {\edge_K^b}, {\edge_L^t},{\edge_L^b} \}$ where ${\edge_K^t}$  (resp. ${\edge_K^b}$) denotes the top (resp. bottom) edge of $K$, as depicted in Figure \ref{fig:full-interpolates}.
Let $v \in \Hmeshizero$; 
 by definition of  $\mathcal R_\edges^{(i,j)} v$, noting that $\left[\frac 1 4 \left(a+b+c+d\right)\right]^q \le a^q+b^q+c^q+d^q$, we have:
\[
  \Vert \mathcal R_\edges^{(i,j)} v \Vert_{L^q(\Omega)}^q \le \sum_{\substack{\edge \in \edgesinti\\ \edge=K|L} } |v_\edge|^q (|D_{\edge_K^t}| + |D_{\edge_K^b}| + |D_{\edge_L^t}| + |D_{\edge_L^b}|)  
   \le 4 \eta^{-2}\sum_{\substack{\edge \in \edgesinti\\ \edge=K|L}} |v_\edge|^q |D_\edge|,
\]
which concludes the proof.
\end{proof}

\begin{figure}[hbt]
\centering
\begin{tikzpicture}[scale=1]
\draw[-](0,0)--(6,0)--(6,2)--(0,2)--(0,0);
\draw[fill=orange!10] (0,0)--(4,0)--(4,2)--(0,2)--(0,0);
%\draw[fill=green!10] (2,0)--(5,0)--(5,2)--(2,2)--(2,0);
\draw[fill=cyan!10] (4,0)--(6,0)--(6,2)--(4,2)--(4,0);
\draw[-] (0,2)--(4,2)--(4,3.5)--(0,3.5)--(0,2);
\draw[-] (4,2)--(6,2)--(6,3.5)--(4,3.5)--(4,2);
\draw[-] (0,-1)--(4,-1)--(4,0)--(0,0)--(0,-1);
\draw[-] (4,-1)--(6,-1)--(6,0)--(4,0)--(4,-1);
\draw[-,thick](4,0)--(4,2);
%
%
% \draw[very thick,green, pattern=north west lines, pattern color=green!30](0,2.75)--(4,2.75)--(4,1)--(0,1)--(0,2.75);%
% \draw[very thick,red, pattern=north east lines, pattern color=red!30](4,2.75)--(4,1)--(6,1)--(6,2.75)--(4,2.75);%
% \draw[very thick,cyan, pattern=north east lines, pattern color=cyan!30](0,1)--(0,-0.5)--(2,-0.5)--(4,-0.5)--(4,1)--(0,1);%
% \draw[very thick,brown, pattern=north west lines, pattern color=brown!30](4,1)--(4,-0.5)--(6,-0.5)--(6,1)--(4,1);%
\draw[very thick,color=green](0,2.75)--(4,2.75)--(4,1)--(0,1)--(0,2.75);%
\draw[very thick,color=red](4,2.75)--(4,1)--(6,1)--(6,2.75)--(4,2.75);%
\draw[very thick,color=cyan](0,1)--(0,-0.5)--(2,-0.5)--(4,-0.5)--(4,1)--(0,1);%
\draw[very thick,color=brown](4,1)--(4,-0.5)--(6,-0.5)--(6,1)--(4,1);%

\path (0.5,1) node[] { $K$};
\path (5.75,1) node[] { $L$};
\path (3.8,0.85) node[rotate=90] { $\edge=K|L$};
%\path (4.5,0.14) node[] { $D_\edge$};
%\path (0.5,3.25) node[] {$M$};
\path (2,1.5) node[green] {$  D_{\edge_K^t}$};
\path (5,1.5) node[red] {$  D_{\edge_L^t}$};
\path (2,0.3) node[cyan] {$  D_{\edge_K^b}$};
\path (5,0.3) node[brown] {$  D_{\edge_L^b}$};
\end{tikzpicture}

\caption{Full grid velocity interpolate.}
\label{fig:full-interpolates}
\end{figure}
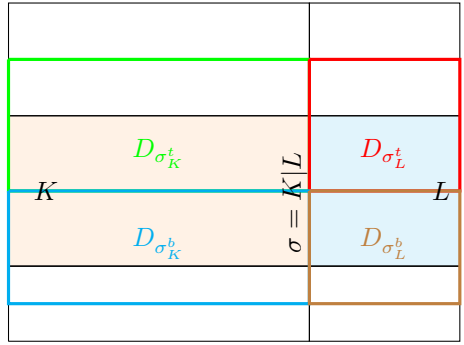
\end{lm}

%The proof of the following Lemma can be found in \cite[Lemma 3.8]{gallouet2015convergence}.
\begin{lm}[Convergence of the full grid velocity interpolate]\label{lem:conv-uchapeau}
  Let $(\mesh_n, \edges_n)_\nnn$ be a sequence of MAC meshes such that $h_{\mesh_n} \to 0$ as $\nti$, and, for all $n$, $\eta_{\mesh_n} \ge \eta >0$.  Let $1 \le q < \infty$.
  
  Let $i,j \in\{1,2,3\}$, $ v \in L^q(\Omega)$ and $(v_n)_\nnn$ be such that $v_n \in H_{\E_n,0}^{(i)}$  and $v_n$ converges to $ v$ as $\nti$  in $L^q(\Omega)$. 
  Let $\mathcal R_{\edges_n}^{(i,j)}$ be the full grid velocity reconstruction operator defined by \eqref{def:ufull}. 
  Then $\mathcal  R_{\edges_n}^{(i,j)} v_n \to  v $ in $L^q(\Omega)$ as $\nti$.
  
  Similarly, if $(v_n)_\nnn$ is such that $v_n \in H_{\E_n}^{(i)}$ and $v_n$ converges to $ v$ as $\nti$  in $L^q(\Omega)$, then,
  $\mathcal R_{\mesh_n}^{(i)}  v_n \to  v $ in $L^q(\Omega)$ as $\nti$, where $\mathcal R_{\mesh_n}^{(i)}  v$ is defined by \eqref{def:ufullprimal}.
\end{lm}
\begin{proof}
 We give the proof for  $\mathcal R_{\edges_n}^{(i,j)}$ (the proof is similar for  $\mathcal R_{\mesh_n}^{(i)}$).
 
 Let $\varphi \in C_c^\infty(\Omega)$. 
  Denoting $\mathcal R_{\edges_n}^{(i,j)}$ by  $\mathcal R_n$ and $\mathcal P_{\edges_n}^{(i)}$ (defined by \eqref{interpedges}) by  $\mathcal P_n$ for short,  we have: 
  \begin{align*}
   \Vert \mathcal  R_n  v_n -  v \Vert_{L^q(\Omega)} \le 
   \Vert  \mathcal  R_n  v_n -\mathcal  R_n \circ \mathcal  P_n v \Vert_{L^q(\Omega)} +
   \Vert   \mathcal  R_n \circ \mathcal P_n v - \mathcal  R_n \circ \mathcal  P_n \varphi \Vert_{L^q(\Omega)} +
   \\
   \Vert  \mathcal  R_n \circ \mathcal P_n \varphi -    \varphi \Vert_{L^q(\Omega)} + \Vert  \mathcal  \varphi - v \Vert_{L^q(\Omega)}.
   \end{align*}
  Since $ \mathcal R_n  v_n = \mathcal R_n \circ \mathcal  P_n v_n$,  and thanks to the fact that
  $\Vert  \mathcal R_n w \Vert_{L^q(\Omega)} \le C \Vert w \Vert_{L^q(\Omega)}$ (for some $C>0$, see Lemma \ref{lem:uchapeau}) and that $\Vert \mathcal P_n w \Vert_{L^q(\Omega)} \le \Vert w \Vert_{L^q(\Omega)}$, we get
   \[ 
   \Vert \mathcal R_n  v_n -v \Vert_{L^q(\Omega)} \le 
   C \Vert v_n - v \Vert_{L^q(\Omega)} +
   C\Vert v - \varphi \Vert_{L^q(\Omega)} +
    \Vert \mathcal  R_n \circ \mathcal P_n \varphi - \varphi \Vert_{L^q(\Omega)} + \Vert  \mathcal  \varphi - v \Vert_{L^q(\Omega)}.
   \]
   Let $\varepsilon >0$. 
   Let us choose $\varphi \in C_c^\infty(\Omega)$ such that $\Vert \mathcal \varphi - v \Vert_{L^q(\Omega)} \le \frac \varepsilon {C+1}$. 
   There exists $n_1$ such that $C \Vert v_n -  v \Vert_{L^q(\Omega)} \le \varepsilon$ for all $n \ge n_1$, and there exists $n_2$ such that $\Vert  \mathcal  R_n \circ \mathcal P_n \varphi - \varphi \Vert_{L^q(\Omega)} \le \varepsilon,$ for all $n \ge n_2$.
   Therefore $\Vert \mathcal R_n  v_n - v \Vert_{L^q(\Omega)} \le 3 \varepsilon$ for $n \ge \max(n_1,n_2)$, which concludes the proof. 
\end{proof}

With the above definitions the following algebraic identity holds (a similar identity is in \cite{gastaldo2014consistent}):
%\begin{lm}\label{alge2}
%Let $\vr \in L_\mesh$ and $\bfu \in \Hmeshzero$.
%Let $\varphi = (\varphi_\edge)_{\edge \in \edgesinti} \in \Hmeshizero$ be a discrete scalar function. Let the primal fluxes $F_{K,\edge}$ be given by ($\ref{eq:massflux}$) and let the dual fluxes $F_{\edge,\edged} $ be given by ($\ref{eq:flux_eK}$) or ($\ref{eq:flux_eorth}$) (depending on the direction of $\bm{e}_i$). For a direction $i \in [1,3]$, an edge $\edge \in \edgesinti \cap \E(K)$
%and  a dual face $\edged \in \edgesd(D_\edge)$, if the vector $\bfe_i$ is normal to $\edged$ and $\edged $ is included in  $K$, we write $\edged=\edged_K$. 
%Then we have for any $1 \le i \le 3$
%\begin{equation*}
%  \sum_{\edge \in \edgesinti} \sum_{\edged \in \edgesd(D_\edge)}
%F_{\edge,\edged} u_{\edged} \varphi_{\edge}= \sum_{K \in \mesh} (\Rr^{(i)}_\mesh \varphi)_K \sum_{j=1}^3 \sum_{\edge \in \edges^{(j)}(K)}F_{K,\edge} (\Rr_{\E}^{(i,j)}{u_i} )_\edge +R^i (u_i,\varphi)
%\end{equation*} 
%where
%\begin{multline*}
%  R^i (u_i,\varphi)=  \sum_{K \in \mesh} \sum_{\edge \in \edges^{(i)}(K)} (\varphi_\edge-(\Rr^{(i)}_\mesh\varphi)_K)  F_{K,\edge}(u_{\edge} - (\Rr^{(i)}_\mesh u_i)_K) \\
%+ \sum_{K \in \mesh} \sum_{\edge \in \edges^{(i)}(K)} (\varphi_\edge- (\Rr^{(i)}_\mesh\varphi)_K) \sum_{\substack{\edged \in \edgesd(D_\edge) \\ \edged \neq \edged_K,\ (\edged \cap \overline{K})\subset \edge'}} \frac{F_{K,\edge'}}{2}(u_{\edged} -(\Rr_\mesh^{(i)}u_i)_K). 
%\end{multline*}
%\end{lm}

\begin{lm}\label{lem-alge2}
Let $\vr \in L_\mesh$ and $\bfu=(u_1,u_2,u_3) \in \Hmeshzero$.
Let $i \in \{1,2,3\}$ and $\varphi = (\varphi_\edge)_{\edge \in \edgesinti} \in \Hmeshizero$ be a discrete scalar function. Let the primal fluxes $F_{K,\edge}$ be given by ($\ref{eq:massflux}$) and let the dual fluxes $F_{\edge,\edged} $ be given by ($\ref{eq:flux_eK}$) or ($\ref{eq:flux_eorth}$).Then we have
\begin{equation*}
  \sum_{\edge \in \edgesinti} \sum_{\edged \in \edgesd(D_\edge)}
F_{\edge,\edged} u_{\edged} \varphi_{\edge}= \sum_{j=1}^3 S_{j},
\end{equation*} 
where
\[
    S_i  = \sum_{\substack{K =  \overrightarrow{[\edge \edge']}\\ \edge, \edge' \in \edgesi}}
    (\vr_\edge^{up} u_\edge |D_{ K,\edge}| + \vr_{\edge'}^{up} u_{\edge'} |D_{ K,\edge'}|)(\Rr_\mesh^{(i)}u_i)_K
    \frac {\varphi_{\edge} - \varphi_{\edge'}}{d(\bfx_\edge, \bfx_{\edge'})},
\]
and, for $j \ne i$,
\[
    S_{j} = \sum_{\substack{\edgeperp \in \edgesintj}} |D_{\edgeperp} | \frac{ \vr^{up}_\edgeperp u_{\edgeperp}}{4} \left[  \left(   {u_{\edge_3}+u_{\edge_1}}\right) \frac {\varphi_{\edge_3} - \varphi_{\edge_1}}{d(\bfx_{\edge_1},\bfx_{\edge_3})} +  \left(   {u_{\edge_4}+u_{\edge_2}}\right) \frac {\varphi_{\edge_4} - \varphi_{\edge_2}}{d(\bfx_{\edge_2},\bfx_{\edge_4})}    \right]
\]
where $(\edge_k)_{k=1,\ldots,4}$ are the four faces (or edges) belonging to $\edgesi$, neighbors of $\edgeperp$, with 
$\bfx_{\edge_3}  \bfx_{\edge_1}=\bfx_{\edge_4}  \bfx_{\edge_2}= \beta \bfe_j$, $\beta>0$ (see Figure \ref{fig-tau}).
 \end{lm}

\begin{proof} We write  $ \sum_{\edge \in \edgesinti} \sum_{\edged \in \edgesd(D_\edge)}
F_{\edge,\edged} u_{\edged} \varphi_{\edge}=    \sum_{j=1}^3 S_{j}$  with, using \eqref{eq:flux_eK}, \eqref{eq:flux_eorth} and the centred choice for $u_{\edged}$,
  \begin{align*}
&   S_i = \sum_{\edge \in \edgesinti} \sum_{\substack{\edged ={\edge | \edge'} \in\edgesdinti \\ \edged\perp\bfe_{i}, \edged \subset K}}
    \frac 1 2 \ \bigl[ F_{K,\edge}\ \bfn_{K,\edge} 
+ F_{K,\edge'}\ \bfn_{K,\edge'}  \bigr] \cdot \bfn_{D_\edge,\edged}
    \frac {u_{\edge} + u_{\edge'}}{2}\varphi_{\edge},
\\
&    S_j =\sum_{\edge \in \edgesinti}  \sum_{\substack{\edged= {\edge | \edge'} \in\edgesdinti\\ \edged \perp\bfe_{j}, \edged \subset \edgeperp\cup\edgeperp'}}  \frac 1 2\ \bigl[F_{K,\edgeperp}+ F_{L,\edgeperp'} \bigr]  \frac {u_{\edge} + u_{\edge'}}{2} \varphi_{\edge},
\textrm{ for } j\ne i,
 \end{align*}
  where $\edgeperp$ and $\edgeperp'$ are the faces of $\edgesj$ such that $\edged \subset \edgeperp \cup \edgeperp'$, $\tau \in \edgesK$, $\tau' \in \edgesL$ and $\edge = K|L$.
  
For $S_i$, a reordering of the summation and the fact that $(u_{\edge} + u_{\edge'})/2 = (\Rr_\mesh^{(i)}u_i)_K$  yield
  \[
    S_i  = \sum_{\substack{K =  \overrightarrow{[\edge \edge']}\\ \edge, \edge' \in \edgesi}}
     \frac 1 2 \ \bigl[ F_{K,\edge'}- F_{K,\edge} \bigr]  (\Rr_\mesh^{(i)}u_i)_K (\varphi_{\edge} - \varphi_{\edge'}).
  \]
Since $F_{K,\edge}=|\edge| \vr_\edge^{up} u_\edge$, this gives
  \[
    S_i  = \sum_{\substack{K =  \overrightarrow{[\edge \edge']}\\ \edge, \edge' \in \edgesi}}
    (\vr_\edge^{up} u_\edge |D_{ K,\edge}| + \vr_{\edge'}^{up} u_{\edge'} |D_{ K,\edge'}|)(\Rr_\mesh^{(i)}u_i)_K
    \frac {\varphi_{\edge} - \varphi_{\edge'}}{d(\bfx_\edge, \bfx_{\edge'})}.
  \]
For $S_j$, $j \ne i$, we have 
%  \begin{multline*}
%    S_{j} = 
%    \sum_{\substack{\edgeperp \in \edgesintj}}   |\edgeperp | \frac{ \vr^{up}_\edgeperp u_{\edgeperp}}{4}   \big[\left(   {u_{\edge_3}+u_{\edge_1}}\right) \varphi_{\edge_1} +  \left(   {u_{\edge_4}+u_{\edge_2}}\right) \varphi_{\edge_2} 
%    \\
%    - \left(   {u_{\edge_1}+u_{\edge_3}}\right) \varphi_{\edge_3}-  \left(   {u_{\edge_2}+u_{\edge_4}}\right) \varphi_{\edge_4}\big]
%    \end{multline*} 
 \begin{equation*}
    S_{j} = 
    \sum_{\substack{\edgeperp \in \edgesintj}}   |\edgeperp | \frac{ \vr^{up}_\edgeperp u_{\edgeperp}}{4}   \big[-\left(   {u_{\edge_3}+u_{\edge_1}}\right) \varphi_{\edge_1} -  \left(   {u_{\edge_4}+u_{\edge_2}}\right) \varphi_{\edge_2} 
    + \left(   {u_{\edge_1}+u_{\edge_3}}\right) \varphi_{\edge_3}+ \left(   {u_{\edge_2}+u_{\edge_4}}\right) \varphi_{\edge_4}\big]
 \end{equation*}
    where $(\edge_k)_{k=1,\ldots,4}$ are the four neighbouring faces (or edges) of $\edgeperp$ belonging to $\edgesi$, \ie \ such that $\bar \edgeperp \cap \bar \edge_k \not = \emptyset$, see figure \ref{fig-tau}.
\begin{figure}[hbt]
\centering
\begin{tikzpicture} 
\draw[thin] (0.5,1.5)--(0.5,3.5);
\draw[thin] (3,1.5)--(3,3.5);
\draw[very thick] (0.5,2.5)--(3,2.5);

\path node at (1.75,2.2) {$\edgeperp$};
\path node at (0.5,2.9) [anchor= west] {$\edge_1$};
\path node at (3,2.9) [anchor= west]{{$\edge_2$}};
\path node at (0.5,2.1) [anchor= west] {$\edge_3$};
\path node at (3,2.1) [anchor= west]{{$\edge_4$}};
\end{tikzpicture}
\caption{Neighbouring faces of $\edgeperp$}
\label{fig-tau}
\end{figure}
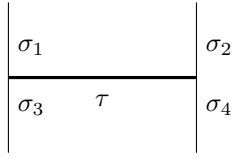

Thus, 
\begin{align*}
    S_{j}  = \sum_{\substack{\edgeperp \in \edgesintj}} |\edgeperp | \frac{ \vr^{up}_\edgeperp u_{\edgeperp}}{4} \left[  \left(   {u_{\edge_3}+u_{\edge_1}}\right) (\varphi_{\edge_3} - \varphi_{\edge_1}) +  \left(   {u_{\edge_4}+u_{\edge_2}}\right) (\varphi_{\edge_4} - \varphi_{\edge_2})\right]
    \\
    = \sum_{\substack{\edgeperp \in \edgesintj}} |D_{\edgeperp} | \frac{ \vr^{up}_\edgeperp u_{\edgeperp}}{4} \left[  \left(   {u_{\edge_3}+u_{\edge_1}}\right) \frac {\varphi_{\edge_3} - \varphi_{\edge_1}}{d(\bfx_{\edge_1},\bfx_{\edge_3})} +  \left(   {u_{\edge_4}+u_{\edge_2}}\right) \frac {\varphi_{\edge_4} - \varphi_{\edge_2}}{d(\bfx_{\edge_2},\bfx_{\edge_4})}    \right]
\end{align*}
\end{proof}

With the uniform estimates stated in Proposition \ref{prop:estimates} and the material introduced above we are able to pass to the limit in the discrete equations ($\ref{dcont}$)--($\ref{dmom}$). 

\begin{prop}
Let $\eta >0$ and $(\disc_n= (\mesh_n,\E_n))_\nnn$ be a sequence of  MAC grids  with step size $h_{\mesh_n}$ tending to zero as $\nti$. Assume that $ \eta \le \eta_{\mesh_n}$ for all $\nnn$, where $\eta_{\mesh_n}$ is defined by ($\ref{regmesh}$).
Let $(\bfu_n)_\nnn$, $(p_n)_\nnn$ and $(\vr_n)_\nnn$  be the corresponding sequence of solutions to ($\ref{probdis}$).
Then, up to the extraction of a subsequence:
\begin{enumerate}
\item the sequence $(\bfu_n)_\nnn$  converges in $(L^q(\Omega))^3$ where $ q \in [1,6) $ to a function $\bfu\in(\xHone_0(\Omega))^3$ and $( \nabla_{\E_n} \bfu_n)_{\nnn} $ converges weakly in $L^2(\Omega)^{3 \times 3} $ to $\nabla \bfu$.
\item the sequence $(\vr_n)_\nnn$ weakly converges to a function $\vr$ in $L^{2\gamma}(\Omega)$,
\item the sequence $(p_n)_\nnn$ weakly converges to a function $p$ in $L^2(\Omega)$,
\item $\bfu$ and $\vr$  satisfy  the continuous mass balance  equation \eqref{contf}.
\item $\bfu$, $p$ and $\vr$ satisfy the continuous momentum balance equation \eqref{movf}.
\item $\vr \ge 0$ a.e. and $\int_\Omega \vr \dx = M$.
\end{enumerate}
\label{prop:conv1}\end{prop}

\begin{proof}
The stated convergences (\ie\ points {\it (1)} to {\it (3)}) are straightforward consequences of the uniform bounds for the sequence of solutions, together, for the velocity, with the compactness theorem \ref{th:compactness} and the Sobolev inequalities stated in Theorem \ref{sobolev}.
Point \textit{(6)} is an easy consequence of point \textit{(2)}. We refer the reader to \cite{eymard2010convergence} for the proof of point \textit{(4)}. Let us then prove point {\it (5)} \textit{i.e.}  that $\bfu$,  $p$ and $\vr$ satisfy  \eqref{movf}.
Let $\bfvphi=(\varphi_1,\varphi_2,\varphi_3)$ be a function of $\xC^\infty_c(\Omega)^3$. 
%Since the support of $\bm{\varphi}$ is compact in $\Omega$, for $n$ large enough, the interpolates of *$\bm{\varphi}$ vanish on the boundary. Hereafter, we assume that we are in this case.
Taking  $\widetilde{\mathcal P}_{\edges_n} \bfvphi   \in \Hmeshnzero$ as a test function in ($\ref{eq:weakmom}$), we infer:
\begin{multline*}
\int_\Omega \dv_{\widetilde\E_n} ( \vr_n \bfu_n \otimes \bfu_n) \cdot \widetilde{\mathcal P}_{\edges_n}\bfvphi   \dx  + \mu \int_\Omega \nabla_{\E_n} \bfu_n :  \nabla_{\E_n}\widetilde{\mathcal P}_{\edges_n}\bfvphi  \dx \\ + (\mu+\lambda) \int_\Omega \dv_{\mesh_n} \bfu_n \dv_{\mesh_n} (\widetilde{\mathcal P}_{\edges_n}\bfvphi  \dx 
-\int_\Omega p_n \dv_{\mesh_n} \widetilde{\mathcal P}_{\edges_n} \bfvphi) \dx = \int_\Omega \mathcal{P}_{\edges_n}\bm{f} \cdot \widetilde{\mathcal P}_{\edges_n}\bfvphi  \dx.
\end{multline*}

The convergence of the diffusive term may be proven by slight modifications of a classical result \cite[Chapter III]{eymard2000finite}:
\[
\lim_\nti
\int_\Omega   \nabla_{\E_n} \bfu_n : \nabla_{\E_n} (\widetilde{\mathcal P}_{\edges_n}\bfvphi)\dx = \int_\Omega \nabla \bfu : \nabla \bfvphi \dx.
\]
By definition of $\widetilde{\mathcal P}_{\edges_n} \bfvphi$  and thanks to Lemma \ref{lem:fortin} we have:
\[
 \int_\Omega  p_n \dv_{\mesh_n}  (\widetilde{\mathcal P}_{\edges_n}\bfvphi)  \dx =
\int_\Omega p_n\ \dive \bfvphi \dx,
\]
and therefore,   thanks to the $L^2$ weak convergence of the pressure,
\[ \lim_\nti \int_\Omega p_n \dv_{\mesh_n} (\widetilde{\mathcal P}_{\edges_n}\bfvphi) \dx = \int_\Omega p\ \dive \bfvphi \dx.
\]
By virtue of the $L^2$ weak convergence of $ \dv_{\mesh_n} \bfu_n$ , we also have:
\begin{equation*}
\lim_\nti \int_\Omega \dv_{\mesh_n} \bfu_n \dv_{\mesh_n}  (\widetilde{\mathcal P}_{\edges_n}\bfvphi) \dx = \int_\Omega \dv \bfu \ \dive \bfvphi \dx.
\end{equation*}

From ($\ref{strongcvfortin}$) and the strong convergence of $ \P_{\E_n} \bm{f} $ towards $\bm{f}$,  we infer that \begin{equation*} \lim_\nti \int_\Omega \mathcal{P}_{\edges_n}\bm{f}  \cdot   \widetilde{\mathcal P}_{\edges_n}\bfvphi \dx= \int_\Omega \bff \cdot \bfvphi \dx. 
\end{equation*}

Now it remains to treat the convective term. Here again  the dependency of the mesh on $n$ will be omitted for short.
 First of all we have
\begin{equation*}
\int_\Omega \dv_{\widetilde\E_n} ( \vr_n \bfu_n \otimes \bfu_n) \cdot  \widetilde{\mathcal P}_{\edges_n}\bfvphi \dx = \sum_{i=1}^3 \sum_{\edge \in \edgesinti} \sum_{\edged \in \edgesd(D_\edge)}
F_{\edge,\edged} u_{\edged}  (\widetilde{\mathcal P}^{(i)}_{\edges_n} \varphi_i)_\edge.
\end{equation*}
Let $1 \le i \le 3 $. Using Lemma \ref{lem-alge2}, we can write, setting $(\widetilde{\mathcal P}^{(i)}_{\edges_n} \varphi_i)_\edge=\psi_\edge$
and using the notations of Lemma \ref{lem-alge2},
\begin{equation*}
 \sum_{\edge \in \edgesinti} \sum_{\edged \in \edgesd(D_\edge)}
F_{\edge,\edged} u_{\edged}  (\widetilde{\mathcal P}^{(i)}_{\edges_n} \varphi_i)_\edge = \sum_{j=1}^3 S_{j},
\end{equation*} 
where
\[
    S_i  = \sum_{\substack{K =  \overrightarrow{[\edge \edge']}\\ \edge, \edge' \in \edgesi}}
    (\vr_\edge^{up} u_\edge |D_{ K,\edge}| + \vr_{\edge'}^{up} u_{\edge'} |D_{ K,\edge'}|)(\Rr_\mesh^{(i)}u_i)_K
    \frac {\psi_{\edge} - \psi_{\edge'}}{d(\bfx_\edge, \bfx_{\edge'})},
\]
and, for $j \ne i$ (see Figure \ref{fig-tau} for the definition of $\sigma_k$, $k=1,\ldots,4$),
\[
    S_{j} = \sum_{\substack{\edgeperp \in \edgesintj}} |D_{\edgeperp} | \frac{ \vr^{up}_\edgeperp u_{\edgeperp}}{4} \left[  \left(   {u_{\edge_3}+u_{\edge_1}}\right) \frac {\varphi_{\edge_3} - \varphi_{\edge_1}}{d(\bfx_{\edge_1},\bfx_{\edge_3})} +  \left(   {u_{\edge_4}+u_{\edge_2}}\right) \frac {\varphi_{\edge_4} - \varphi_{\edge_2}}{d(\bfx_{\edge_2},\bfx_{\edge_4})}    \right].
\]
Replacing, in $S_i$, $\vr^{up}_\edge$ by $\vr_K$,  the term $S_i$ can be written as $S_i=\bar S_i+ R_i$ with
\[
\bar S_i  = \sum_{\substack{K =  \overrightarrow{[\edge \edge']}\\ \edge, \edge' \in \edgesi}}
    (\vr_K u_\edge |D_{ K,\edge}| + \vr_K u_{\edge'} |D_{ K,\edge'}|)(\Rr_\mesh^{(i)}u_i)_K
    \frac {\psi_{\edge} - \psi_{\edge'}}{d(\bfx_\edge, \bfx_{\edge'})}.
\]
Thanks to the weak convergence of $\vr$ in $L^2(\Omega)$, the convergence of $\bfu$ in $L^4(\Omega)^3$, Lemma \ref{lem:conv-uchapeau} and the uniform
convergence of the term $ \frac {\psi_{\edge} - \psi_{\edge'}}{d(\bfx_\edge, \bfx_{\edge'})}$ to $-\partial_i \vphi_i$, we obtain
\[
\lim_\nti \bar S_i = -\int_\Omega \vr u_i u_i \partial_i \vphi_i \dx.
\]
Furthermore, using H\"older's inequality and Inequality \eqref{bvweak}, one has $| R_i| \le C \sqrt{h_{\mesh_n}}$ and then 
\[
\lim_\nti S_i = -\int_\Omega \vr u_i u_i \partial_i \vphi_i \dx.
\]
For $j \ne i$ we can write $S_j=\bar S_j + R_j$ with
\begin{multline*}
\bar S_j=-\sum_{\substack{\edgeperp \in \edgesintj}} |D_{\edgeperp} | \frac{ \vr^{up}_\edgeperp u_{\edgeperp}}{4} \left[  \left(   {u_{\edge_3}+u_{\edge_1}}\right) \partial_j \vphi_i(\bfx_\edgeperp) +  \left(   {u_{\edge_4}+u_{\edge_2}}\right) \partial_j \vphi_i(\bfx_\edgeperp)  \right]
\\=
-\sum_{\substack{\edgeperp \in \edgesintj}} |D_{\edgeperp} |  \vr^{up}_\edgeperp u_{\edgeperp}  (\Rr_{\E_n}^{(i,j)} {u_i})_{\edgeperp} \partial_j \vphi_i(\bfx_\edgeperp),
\end{multline*}
and $|R_j| \le C h_{\mesh_n}$ thanks to the $L^2$-bound for $\vr$, the $L^4$-bound for $\bfu$, Lemma \ref{lem:uchapeau} and the
regularity of $\vphi_i$.

Now, as for $S_i$, we replace $\vr^{up}_\edgeperp$ by $\vr_K$ or $\vr_L$ (for $\edgeperp=K|L$),  the term $\bar S_j$ can be written as $\bar S_j=\tilde S_j+ \tilde R_j$ with
\[
\tilde S_j=-
\sum_{\substack{\edgeperp \in \edgesintj}} (|D_{K,\edgeperp} | \vr_K +  |D_{L,\edgeperp}|\vr_L)u_{\edgeperp} (\Rr_{\E_n}^{(i,j)} {u_i})_{\edgeperp} \partial_j \vphi_i(\bfx_\edgeperp).
\]
As for $\bar S_i$ (weak convergence $\vr$ in $L^2(\Omega)$, convergence of $\bfu$ in $L^4(\Omega)^3$, Lemma \ref{lem:conv-uchapeau} and regularity of $\vphi_i$), we obtain
\[
\lim_\nti \tilde S_j= -\int_\Omega \vr u_i u_j \partial_j \vphi_i \dx.
\]
Furthermore, using H\"older's inequality and Inequality \eqref{bvweak}, one has $| \tilde R_j| \le C \sqrt{h_{\mesh_n}}$ and then 
\[
\lim_\nti S_j = -\int_\Omega \vr u_i u_j \partial_j \vphi_i \dx.
\]
Summing the limit of $S_j$ for $j=1,2,3$, we obtain
\[
\lim_\nti \sum_{\edge \in \edgesinti} \sum_{\edged \in \edgesd(D_\edge)}
F_{\edge,\edged} u_{\edged}  (\widetilde{\mathcal P}^{(i)}_{\edges_n} \varphi_i)_\edge=
-\int_\Omega u_i \vr \bfu \cdot \grad \vphi_i \dx.
\]
Now, 
summing for $i \in \{1,2,3\}$ we obtain
\begin{equation*}
\int_\Omega \dv_{\widetilde\E_n} (\vr_n \bfu_n \otimes \bfu_n) \cdot \bvarphi \dx \to - \int_\Omega \vr \bfu \otimes \bfu : \nabla \bm{\varphi} \dx \ \text{as} \ \nti.
\end{equation*}
Finally $\bfu,p,\vr $ satisfy point {\it (5)} and the proof of Proposition \ref{prop:conv1} is complete.
\end{proof}

\subsection{Passing to the limit in the equation of the state}

The goal of this part is to pass to the limit in the nonlinear equation ($\ref{deos}$). As in the continuous case, the main idea is to prove the a.e. convergence of $\vr_n $ towards $\vr$ (up to a subsequence).

\subsubsection{The effective viscous flux}

To overtake this difficulty in the continuous case we have proved that the sequence of approximate solution satisfy ($\ref{palnl}$). 
The following proposition is nothing else than the discrete version of this identity.

\begin{prop}[Convergence of the effective viscous flux]
Under the assumptions of Proposition \ref{prop:conv1}
we have  for all $\varphi \in \xC^\infty_c(\Omega)$,
\begin{equation}\label{fevd}
\lim_{\nti} \int_\Omega (p_n -(\lambda+2\mu) \divedn\,  \bfu_n) \vr_n \varphi \dx = \int_\Omega (p-(\lambda+2\mu) \dive\,  \bfu) \vr \varphi \dx,
\end{equation}
passing to subsequences if necessary.
\label{peff} \end{prop}

\begin{proof}

The following proof can be seen as a discrete version of Step 3 of the proof of Theorem \ref{continuityws}.

Let $\varphi \in \xC^\infty_c(\Omega)$.
For a MAC grid $\mesh$, we define $\varphi_\mesh \in L_\mesh$,  $\varphi^{(i)}_{\E} \in H_{\E,0}^{(i)}$ by:
\begin{equation*}
\left\{
\begin{array}{l}
 \varphi_\mesh (\bfx) = \varphi (\bfx_K), \forall \bfx \in K, \ \forall K \in \mesh, \\ \\
  \varphi_{\E}^{(i)} (\bfx) = \varphi (\bfx_\edge), \forall \bfx \in D_\edge, \ \forall \edge \in \E^{(i)}.
\end{array}
\right.
\end{equation*}

%For a sequence of grids $\mesh_n$, for short we shall denote $\varphi_n = \varphi_{\mesh_n}$.
We define $w_n$ with \eqref{fvle} (with $\mesh_n$ and $\vr_n$ instead of $\mesh$ and $\vr$) and $\bfv_n$ with $ \bv_n = -\gradtg_{\E_n} w_n $.
We set $ \bV_n =(V_{n,1},V_{n,2},V_{n,3})= ( v_{n,1} \varphi_{\E_n}^{(1)},v_{n,2} \varphi_{\E_n}^{(2)} ,v_{n,3} \varphi_{\E_n}^{(3)}) $.

Thanks to Lemma~\ref{locest}, since $\vr_n$ is bounded in $L^2(\Omega)$, the compactness theorem~\ref{th:compactness} gives that, up to a subsequence, as $n \tends \infty$, $\bfv_n$ converges to some $\bfv=(v_1,v_2,v_3)$ in $L^2_{loc}(\Omega)^3$ and that $\bfv \in \xHone_{loc}(\Omega)^3$. As a consequence, using  Theorem \ref{sobolev}, the sequence $(\bV_n)_{\nnn} $ converges to $ \varphi \bv $ in $ L^q(\Omega)^3 $ for any $ q \in [1,6)$.
As a consequence of the compactness theorem~\ref{th:compactness} we also have that $\dived \, \bfu_n$ and $\curld   \bfu_n$ converge weakly in $L^2(\Omega)$ towards $\dive \bfu$ and $\curl \bfu$.

Since $ \bV_n  \in \Hmeshnzero$, it is possible to take $\bV_n$ in \eqref{eq:weakmom} and write, using Lemma \ref{ggddccd}, 
\begin{multline}
\int_\Omega \dv_{\widetilde\E_n}(\vr_n \bfu_n \otimes \bfu_n) \cdot \bV_n \dx + (\lambda+2\mu) \int_\Omega \dv_{\mesh_n}  \bfu_n \, \dv_{\mesh_n}   \bV_n \dx
\\+\mu \int_\Omega \curl_{\mesh_n}  \bfu_n  \cdot \curl_{\mesh_n}  \bV_n \dx
 - \int_\Omega  p_n \dv_{\mesh_n}   \bV_n \dx
= \int _\Omega {\mathcal{P}}_{\E_n} \bm{f} \cdot  \bV_n \dx.
\label{lim} \end{multline}
where we have used formula ($\ref{gdcd}$).
We now mimick the proof given in the continuous case for the proof of \eqref{palnlff}. 
Since $\dv_{\mesh_n}  \bfv_n=\vr_n$, we first remark that:
\begin{equation}
\int_\Omega  \dv_{\mesh_n}  \bfu_n \, \dv_{\mesh_n}   \bV_n \dx = 
\\
\int_\Omega  (\dv_{\mesh_n} \bfu_n)  \vr_n \varphi \dx 
+ \int_\Omega ( \dv_{\mesh_n} \bfu_n) \bfv_n \cdot \gradi \varphi \dx +R_{1,n},
\label{limd}
\end{equation}
where $\lim_\nti R_{1,n} =0$, thanks to the discrete $\xHone(\Omega)$-estimate \eqref{estimup} on $\bfu_n$ and the $L^2_{loc}(\Omega)$ estimate of Lemma \ref{locest} on $\bfv_n$.
Replacing $\dv_{\mesh_n}\bfu_n$ by $p_n$, the same computation gives:
\begin{equation}
\int_\Omega p_n \,  \dv_{\mesh_n}  \bV_n \dx = \int_\Omega  p_n  \vr_n \varphi \dx 
+ \int_\Omega p_n \bfv_n \cdot \gradi \varphi \dx +R_{2,n},
\label{limp}\end{equation}
where $\lim_\nti R_{2,n} =0$.
In accordance with \cite{eymard2010convergence},  the second term of \eqref{lim} can be transformed as follows:
\begin{equation}
\begin{array}{l} \displaystyle
\int_\Omega  \curl_{\mesh_n} \bfu_n \cdot  \curl_{\mesh_n}  \bV_n \dx =
\int_\Omega \curl_{\mesh_n} \bfu_n \cdot \curl_{\mesh_n} \bfv_n \;  \varphi \dx 
\\ \displaystyle \hspace{20ex}
+ \int_\Omega \curl_{\mesh_n}  \bfu_n \cdot L(\varphi) {\overline{\bfv}_n} \dx +R_{3,n},
\end{array}
\label{limt}
\end{equation}
where $\lim_\nti R_{3,n} =0$ (for the same reasons as $R_{1,n}$), the matrix $L(\varphi)$ is the same as in the proof of \eqref{palnlff} and involves the first order derivatives of $\varphi$,  and $\overline{\bfv}_n$  satisfies:
\begin{equation}\label{convloc}
\overline{\bfv}_n \to \bfv~ \text{in}~ L^2_{\text{loc}}(\Omega)^3~\text{as}~ n \to +\infty.
\end{equation}
We refer the interested reader to \cite{eymard2010convergence} for an explicit expression of $\overline{\bfv}_n$ and for a proof of ($\ref{convloc}$).

Since $\curl_{\mesh_n}  \bfv_n=0$, \eqref{limt} leads to:
\begin{equation}
\int_\Omega  \curl_{\mesh_n}  \bfu_n \cdot  \curl_{\mesh_n}  \bV_n \dx =
 \int_\Omega \curl_{\mesh_n} \bfu_n \cdot L(\varphi) {\overline{\bfv}_n} \dx +R_{3,n}.
\label{limc}\end{equation}
Let us turn our attention to the convective term. For the readability, the dependency of some terms with respect to $n$ will be omitted when there are indices
related to the mesh (such as $\edge$, $\edged$, $\edgeperp$).

One has
\[
 \int_\Omega \dv_{\widetilde\E_n}(\vr_n \bfu_n \otimes \bfu_n) \cdot \bV_n \dx = \sum_{i=1}^3 \sum_{\edge \in \edgesinti} \sum_{\edged \in \edgesd(D_\edge)}
F_{\edge,\edged} u_{\edged} V_{\edge},
\]
where $V_{\edge}$ is the value of $V_{n,i}$ in $D_\edge$.
Let $i \in \{1,2,3\}$.
Setting $Q_n=\sum_{\edge \in \edgesinti} Q_\edge \characteristic_{D_\edge}$ with $Q_\edge=(1/|D_\edge|)\sum_{\edged \in \edgesd(D_\edge)}
F_{\edge,\edged} u_{\edged}$, one has 
\begin{equation}
\sum_{\edge \in \edgesinti} \sum_{\edged \in \edgesd(D_\edge)}
F_{\edge,\edged} u_{\edged} V_{\edge}= \int_\Omega Q_n V_{n,i} \dx. \label{Qn}
\end{equation}
We recall that $V_{n,i} \to \vphi v_i$  in $L^q(\Omega)$ for $q<6$ (as $\nti$).
In a first step, we prove that the sequence $(Q_n)_\nnn$ is bounded in $L^p(\Omega)$ for some $p >6/5$
(indeed we will have $p$ such that $1/p= 1/(2\g) + 1/2 + 1/6$ and then $p > 6/5$ since $\g>3$).
Then, up to subsequence, $Q_n \to Q$ weakly in $L^p(\Omega)$. In a second step we identify $Q$, proving that
$Q=\vr \sum_{j=1}^3 u_j \partial_j u_i$.

{\bf - Estimate on $Q_n$. }
For $\edge \in \edgesinti$, we use \eqref{eq:mass_D_imp}. 
It gives
\begin{equation}
Q_\edge=\frac{1}{|D_\edge|}\sum_{\edged \in \edgesd(D_\edge)}
F_{\edge,\edged} (u_{\edged}-u_\edge) - \cstab h_\mesh^\alpha( \vr_{\Ds} - \vr^\star)u_\edge.
\label{cucu}
\end{equation}
Let $\edged \in \edgesd(D_\edge)$ such that $\edged = \edge | \edge' \in\edgesdinti$
\begin{itemize}
\item If $\edged\perp\bfe_{i}$, $\edged \subset K$, then
\[
| F_{\edge,\edged}| \le \frac 1 2 (|F_{K,\edge}|+|F_{K,\edge'}| )\le 
\frac 1 2 (|\edge| \vr_{\edge}^{up} |u_{\edge}| + |\edge' | \vr_{\edge'}^{up} |u_{\edge'}|).
\]
\item If   $\edged \perp\bfe_{j}$, $j \ne i$, $\edged \subset \edgeperp\cup\edgeperp'$,
 where $\edgeperp$ and $\edgeperp'$ are the faces of $\edgesj$ such that $\edged \subset \edgeperp \cup \edgeperp'$, $\tau \in \edgesK$, $\tau' \in \edgesL$, $\edge=K|L$, then,
 \[
| F_{\edge,\edged}| \le \frac 1 2 (|F_{K,\edgeperp}|+ |F_{L,\edgeperp'|})\le 
\frac 1 2 (|\edgeperp| \vr_{\edgeperp}^{up} |u_{\edgeperp}| + |\edgeperp' | \vr_{\edgeperp'}^{up} |u_{\edgeperp'}|).
\]
\end{itemize}
Using the estimates on $\vr$ in $L^{2\g}(\Omega)$, $\bfu$ in $L^6(\Omega)$, $\grad_{\widetilde\E} u_i$ in $L^2(\Omega)$ and the fact that $\eta_n \ge \eta$ for all $n$, the part of $Q$ given by the first term of \refe{cucu} is bounded in $L^p(\Omega)$ with $p$ such that $1/p= 1/(2\g) + 1/2 + 1/6$.
The part of $Q$ given by the second term of \refe{cucu} tends to $0$ in $L^{3/2}(\Omega)$ for instance (since $\vr$ is bounded in $L^{2}(\Omega)$ and $\bfu$ in $L^6(\Omega)$) and then also in $L^p(\Omega)$.
Thus, up to a subsequence, we can assume that $Q_n \to Q$ weakly in $L^p(\Omega)$ and this gives
\begin{equation}
\lim_{\nti}\int_\Omega Q_n V_{n,i}\dx = \int_\Omega Q \vphi v_i \dx.
\end{equation}

{\bf  - Identification of  $Q$. }
Let $\bar \vphi \in C_c^\infty(\Omega)$. 
For $\edge \in \edgesinti$, let $\bar \vphi_\edge=(\widetilde{\mathcal P}^{(i)}_{\edges_n} \bar \varphi)_\edge$.
Then,  for $h_n$ small enough,
\[
\int_\Omega Q_n \bar \vphi \dx = \sum_{\edge \in \edgesinti} \sum_{\edged \in \edgesd(D_\edge)}
F_{\edge,\edged} u_{\edged}\bar \vphi_{\edge}.
\]
We already passed to the limit on this term in Proposition \ref{prop:conv1}:
\[
\lim_\nti \sum_{\edge \in \edgesinti} \sum_{\edged \in \edgesd(D_\edge)}
F_{\edge,\edged} u_{\edged}  (\widetilde{\mathcal P}^{(i)}_{\edges_n} \bar\varphi)_\edge=
-\int_\Omega u_i \vr \bfu \cdot \grad \bar \vphi \dx.
\]
Then $\int_\Omega Q \bar \vphi \dx = -\int_\Omega u_i \vr \bfu \cdot \grad \bar \vphi \dx$.
Since we already know that $\div ( \vr \bfu)=0$ we obtain (using $u_i \in H^1(\Omega)$ and $\vr \bfu \in L^2(\Omega)^3$)
\[
Q = \sum_{j=1}^3 \vr u_j \partial_j u_i.
\]
Finally, we have the limit of the convection term:
\begin{equation}
\lim_\nti \int_\Omega \dv_{\widetilde\E_n}(\vr_n \bfu_n \otimes \bfu_n) \cdot \bV_n \dx = \int_\Omega \sum_{i=1}^3
 \sum_{j=1}^3 \vr u_j (\partial_j u_i) \vphi v_i \dx.
\label{weaklimitfev}
\end{equation}
We recall now that $(\bV_n)_{\nnn} $ converges to $ \varphi \bv $ in $ L^q(\Omega)^3 $ for any $ q \in [1,6)$ and that $\dive_{\mesh_n} \bfu_n$, $p_n$ and  $\curl_{\mesh_n}\bfu_n$ weakly converge respectively in $L^2(\Omega)$ and $L^2(\Omega)^3$ to $\dive \bfu$, $p$ and $\curl  \bfu$.
Then, using \eqref{limd}--\eqref{limc}, we deduce from \eqref{lim} and \eqref{weaklimitfev}:
\begin{multline*}
\lim_{\nti} \int_\Omega \Big( (\lambda+2\mu)\dived  \bfu_n-p_n\Big)\,\vr_n \varphi \dx = \int_\Omega  \Big(p-(\lambda+2\mu)\dive  \bfu\Big)\, \bfv \cdot \gradi \varphi \dx \\  - \mu \int_\Omega  \curl  \bfu \cdot (L(\varphi) \bfv) \dx -\int_\Omega \vr( (\bfu \cdot \nabla) \bfu)\cdot \varphi \bfv \dx + \int_\Omega \bff\cdot \bfv \varphi  \dx.
\end{multline*}
Finally, since $p_n$ and $\bfu_n$ are solution of the discrete momentum balance equations, we already know, thanks to the estimates on $p_n$ and $\vr_n$, that the limits $p$ and $\bfu$ are solution of the momentum balance equation; hence, since $\bfv \in \xHone_{loc}(\Omega)^3$ and in accordance with the continuous case:
\[
\begin{array}{l}
\displaystyle \int_\Omega \Big((2\mu+\lambda)\dive \bfu-p\Big)\, (\dive  \bfv) \,  \varphi \dx - \int_\Omega \vr \bfu \otimes \bu : \nabla (\varphi \bfv) \dx =
\\ \hfill
\displaystyle \int_\Omega \Bigl( (p-(2\mu+\lambda)\dive  \bfu)\, \bfv \cdot \gradi \varphi 
- \mu \curl  \bfu \cdot (L(\varphi) \bfv)  - \mu \curl  \bfu \cdot \curl \bfv  \varphi + \bff \cdot \bfv \varphi  \Bigr) \dx.
\end{array}
\]
Moreover we know that $\dv(\vr \bfu) = 0 $ and $(\vr,\bfu) \in L^6(\Omega) \times H_0^1(\Omega)^3 $ and consequently $ \int_\Omega \vr \bfu \otimes \bfu : \nabla (\varphi \bfv) \dx = -\int_\Omega \vr( \bfu \cdot \nabla \bfu)\cdot \varphi \bfv \dx$.
Since  $\dive_{\mesh_n} \bfv_n$ and $\curl_{\mesh_n} \bfv_n$ converge weakly in $L^2_{loc}(\Omega)$ towards $\dive \bfv$ and $\curl\bfv$, one has $\dive \bfv=\vr$ and $\curl  \bfv=0$ and therefore:
\begin{multline*}
\int_\Omega \Big((2\mu+\lambda)\dive  \bfu-p\Big)\, \vr \varphi \dx =
\int_\Omega \Big( (p-(2\mu+\lambda)\dive  \bfu)\, \bfv \cdot \gradi \varphi 
- \mu(\curl   \bfu)\cdot L(\varphi) \bfv \Big) \dx. \\ 
-\int_\Omega \vr( \bfu \cdot \nabla \bfu)\cdot \varphi \bfv\dx +\int_\Omega \bff\cdot \bfv \varphi  \dx.
\end{multline*}
Then, we obtain the desired result, that is:
\begin{equation}
 \lim_{\nti} \int_\Omega ( p_n -(\lambda+2\mu) \dv_{\mesh_n}  \bfu_n)\, \vr_n \varphi \dx= \int_\Omega (p -(\lambda+2\mu) \dive \bfu)\, \vr \varphi \dx.
\label{stromboli}\end{equation}
\end{proof}

\subsubsection{A.e. and strong convergence of \texorpdfstring{$ \vr_n$ and $p_n$}{Lg}}
Let us now prove the a.e. convergence of $\vr_n$ and $ p_n$.
Using \cite[Lemma 2.1]{eymard2010convergent}, one can take $\varphi=1$ in ($\ref{fevd}$), wich gives:
\begin{equation*}
\lim_{\nti} \int_\Omega (p_n -(2\mu+\lambda)\dv_{\mesh_n} \bfu_n)\vr_n \dx = \int_\Omega (p-(2\mu+\lambda)\dv_\mesh \bfu)\vr \dx
\end{equation*}
Now using Lemma \ref{lmm:renorm} and ($\ref{cruun}$) we obtain the discrete version of ($\ref{palnlff}$) that is
\begin{equation}\label{convprho}
\limsup_{\nti} \int_\Omega p_n \vr_n \dx \le \int_\Omega p \vr \dx.
\end{equation}
Let $G_n= ( \vr_n^\gamma - \vr^\gamma)( \vr_n - \vr)$.
One has $G_n \in  L^1(\Omega)$ and $G_n \ge 0$ a.e. in $\Omega$.
Futhermore:
\[
\int_\Omega G_n \dx  = 
\int_\Omega  p_n  \vr_n \dx - \int_\Omega  p_n \vr \dx
- \int_\Omega  \vr^\gamma  \vr_n \dx  + \int_\Omega \vr^\gamma \vr \dx.
\]
Using the weak convergence in $L^2(\Omega)$ of $p_n$ and $ \vr_n$, and \eqref{convprho}, we obtain:
\[
\limsup_{\nti}  \int_\Omega G_n \dx  \le 0.
\]
Then (up to a subsequence), $G_n \tends 0$ a.e. and then $\vr_n \tends \vr$ a.e. (since $y \mapsto y^\gamma$ is an increasing function on $\xR_+$).
Finally, $\vr_n \tends \vr \textrm{ in } L^q(\Omega) \textrm{ for all } 1 \le q < 2\gamma$, $ p_n= \vr_n^{\gamma}  \tends \vr^\gamma \textrm{ in } L^q(\Omega) \textrm{ for all } 1 \le q < 2$, and $p=\vr^\gamma$.
We have thus proved the convergence of the approximate pressure and density, which, together with Proposition \ref{prop:conv1}, concludes the proof of Theorem \ref{theo:conv1}.

\section{Conclusion}
In this paper, we considered the MAC scheme for the stationary baro\-tro\-pic compressible Navier-Stokes equations and proved its convergence in the case $\gamma >3$. 
This latter restriction on $\gamma$ is used when writing the nonlinear convection term as in \eqref{conv-rewrite} in order to prove its convergence in the continuous case, in a manner that adapts to the discrete case, which is the case here with the convergence of $Q_n$ in \eqref{Qn}. 
So far, it is an open question to find a  technique of convergence of the nonlinear convection term that would adapt to the discrete case without requiring this condition.
%This scheme, which is very popular in the computational fluid dynamics community, is also proved to be quite adapted to a convergence analysis.
%%To our knowledge, the convergence analysis established in this article seems to be the first for this problem.
%Ongoing work concerns the extension to the non stationary  Navier- Stokes equations in two or three space dimensions.
%
\appendix

\section{Existence of a discrete solution}\label{existproof}

This section is devoted to the proof of Theorem \ref{thmexist}.
We now state the abstract theorem which will be used hereafter.

\begin{thm}\label{topologicaldegreethm}
Let $N$ and $M$ be two positive integers and $V$ be defined as follows:
\begin{equation*}
V = \{ (x,y) \in \R^N \times \R^M, \ y>0 \},
\end{equation*}
where, for any real number $c$, the notation $y > c$ is meant componentwise. Let  $ F $ be a continuous function from  $V \times [0,1]$ to $\R^N \times \R^M$ satisfying:
\begin{enumerate}
\item $\forall \zeta \in [0,1] $, if $ v \in V $ is such that $ F(v,\zeta)=0 $ then $ v \in W $ where $W  = \{ (x,y) \in \R^N \times \R^M, \ \| x\| < C_1, \ \text{and} \ \varepsilon < y < C_2 \}$,
with $C_1$ , $C_2,$ and $\varepsilon >0$  and $ \| \cdot \| $ a norm defined over $\R^N$ ;
\item the topological degree of $F (\cdot, 0)$ with respect to $0$ and $W$ is equal to $d_0 \ne 0$.
\end{enumerate}
Then the topological degree of $F (\cdot, 1)$ with respect to $0$ and $W$ is also equal to $d_0  \ne 0$; consequently, there exists
at least a solution $ v \in W$ such that $F(v,1) = 0$.
\end{thm}

Let us now prove the existence of a solution to \eqref{probdis}. 
Let us define 
\begin{equation*}
V =\{ (\bfu,\vr) \in \Hmeshzero \times L_\mesh, \ \vr_K > 0 \ \forall K \in \mesh \}.
\end{equation*}
and consider the continuous mapping 
  \begin{align*}
   F : &\ \Hmeshzero \times L_\mesh  \times [0,1]\longrightarrow \Hmeshzero \times L_\mesh \\
       &\ (\bfu, \vr,\zeta)\mapsto F(\bfu,\vr,\zeta) = (\hat{\bfu},\hat \vr)
  \end{align*}
  where $(\hat{\bfu},\hat{\vr})$ is the unique element of $\Hmeshzero\times L_{\mesh}$ such that 
%  \begin{subequations}
%    \begin{align}
%    & \int_{\Omega}\hat{\bfu}\cdot\bfv \dx=\mu [\bfu,\bfv]_{1,\edges,0}+ (\mu+\lambda) \int_\Omega \dv_\mesh \bfu \dv_\mesh \bfv \dx +\zeta \ \int_\Omega \dv_{\E} (\vr \bfu \otimes \bfu) \cdot \bfv \dx     \nonumber\\
%& \quad \quad \quad \quad \quad \quad \quad \quad \quad \quad \quad \quad \quad \quad -\zeta\int_{\Omega} \ \vr^\gamma \ \dive_{\mesh} \ \bfv -\int_{\Omega}\P_{\edges} \bm{f} \cdot\bfv\dx, \ \forall\bfv\in\Hmeshzero, \label{F1}  \\
%    & \int_{\Omega}\hat{\vr} \ q \dx= \zeta\int_{\Omega}  \dive_{\mesh}^{\upw} (\vr \bfu) \ q\dx+ \int_\Omega \cstab h_\mesh^\alpha(\vr-\vr^\star)q \dx, \ \forall q\in L_{\mesh}.
%    \label{F2} \\ 
% \end{align}
%   \end{subequations}
%\begin{subequations}
    \begin{align}
&     \int_{\Omega}\hat{\bfu}\cdot\bfv \dx=\mu [\bfu,\bfv]_{1,\edges,0}+ (\mu+\lambda) \int_\Omega \dv_\mesh \bfu \dv_\mesh \bfv \dx
  \nonumber
    \\
&   +\zeta\int_\Omega \dv_{\widetilde\E} (\vr \bfu \otimes \bfu) \cdot \bfv \dx    
   -\zeta\int_{\Omega} \ \vr^\gamma \ \dive_{\mesh}  \bfv -\int_{\Omega}\P_{\edges} \bm{f} \cdot\bfv\dx, \ \forall\bfv\in\Hmeshzero, \label{F1} 
   \\
&    \int_{\Omega}\hat{\vr} \ q \dx= \zeta\int_{\Omega}  \dive_{\mesh}^{\upw} (\vr \bfu) \ q\dx+ \int_\Omega \cstab h_\mesh^\alpha(\vr-\vr^\star)q \dx, \ \forall q\in L_{\mesh}.
   \label{F2}
  \end{align}
%\end{subequations}
Note that  the values of $\hat{u}_i$, $ i=1,\cdots,d$, and $\hat{\vr}$ are readily obtained by setting in this system $\vi=1_{D_{\edge}}$, $v_j=0, j\ne i$ in \eqref{F1} and $q=1_{K}$ in \eqref{F2}.

Any  solution of $F(\bfu,\vr,1)=0$ is a solution of Problem \ref{probdis} where $p = \vr^\gamma $.
   
The mapping $F$ is continuous. 

Let $(\bfu,\vr)\in \Hmeshzero\times L_{\mesh} $ and $\zeta \in [0,1]$  such that $F(\bfu,\vr,\zeta)=(0,0)$ (in particular $\vr >0$). Then for any $ (\bfv,q) \in \Hmeshzero\times L_\mesh$,
  \begin{subequations}\label{weak:ro}
  \begin{align} 
      &\zeta \int_\Omega \dv_{\widetilde\E}(\vr \bfu \otimes \bfu) \dx+  \mu [\bfu,\bfv]_{1,\edges,0} + (\mu+\lambda)\int_\Omega \dv_\mesh \bfu \dv_\mesh \bfv \dx \nonumber \\
&  \quad \quad \quad \quad \quad \quad \quad \quad \quad \quad -\zeta \int_\Omega \vr^\gamma\, \dv_\mesh \bv \dx = \int_\Omega \mathcal P_\edges \bff \cdot \bfv \dx, \label{dmomzeta}
      \\  
      & \zeta\int_\Omega \dive_\mesh^{\upw}  (\vr \bfu)\, q \dx +  \int_\Omega \cstab h_\mesh^\alpha(\vr-\vr^\star)q \dx =0. \label{dcontzeta}
    \end{align}
  \end{subequations}
Taking $q=1$ as a test function in ($\ref{dcontzeta}$), and using the conservativity of  the fluxes we obtain 
\begin{equation}\label{massconserv}
\int_\Omega \vr \dx= \| \vr \|_{L^1(\Omega)} = M >0.
\end{equation}
This relation provides a bound for $\vr$ in the $L^1$ norm, and therefore in all norms since the problem is of finite dimension. 
Taking $\bfu $ as a test function in ($\ref{dmomzeta}$) and following Step 1 of the proof of Proposition \ref{prop:estimates} gives
\begin{equation}\label{estvelocity}
 \| \bu \|_{1,\E,0} < C_1
\end{equation}
where the constant $C_1$ depends only on the data of the problem and not on $\zeta$. Now a straightforward computation gives
\begin{equation*}
\vr_K \ge \frac{ \cstab \min_{L \in \mesh} |L| h_\mesh^\alpha \vr^\star}{\cstab h_\mesh^\alpha |\Omega| + \stkl |\edge||u_{K,\edge} |  }
\end{equation*}
Consequently by virtue of  ($\ref{estvelocity}$) there exists $\varepsilon > 0 $ such that 
\begin{equation}\label{estunderrho}
\vr_K > \varepsilon, \ \forall K \in \mesh,
\end{equation}
where the constant $\varepsilon$ depends only on the data of the problem. Clearly, from ($\ref{massconserv}$), one has also
\begin{equation}\label{estrhoabove}
\vr_K \le \frac{M}{ \min_{K \in \mesh} |K|}=C_2-1, \ \forall K \in \mesh.
\end{equation}
Moreover the system $F(\bfu,\vr,0) = 0 $ reads:
  \begin{subequations}\label{weak:ro0}
  \begin{align} 
      &  \mu [\bfu,\bfv]_{1,\edges,0} + (\mu+\lambda)\int_\Omega \dv_\mesh \bfu \dv_\mesh \bfv \dx  = \int_\Omega \mathcal P_\edges \bff \cdot \bfv \dx, \ \forall \bv \in \Hmeshzero, \label{dmomzeta0}
      \\  
      &  \vr_K =\vr^\star, \ \forall K \in \mesh. \label{dcontzeta0}
    \end{align}
  \end{subequations}
which has clearly one and only one solution.
Let W defined by
\begin{equation*}
W = \{ (\bfu,\vr) \in \Hmeshzero \times L_\mesh \ \text{such that} \ \| \bfu \| < C_1, \ \varepsilon < \vr_K < C_2 \}
\end{equation*}
%It is quite easy to see that the jacobian matrix does not vanish for the solution of the system ($\ref{weak:ro0}$). Therefore
Since $F(\bfu,\vr,0) = 0 $ is a linear system which has one and only one  solution belonging to $W$, the topological degree $d_0$ of $F (\cdot, \cdot, 0)$ with respect to 0 and $W$ is not zero. Then, using the inequalities  ($\ref{estvelocity}$), ($\ref{estunderrho}$), ($\ref{estrhoabove}$), Theorem \ref{topologicaldegreethm} applies, which concludes the proof.

\bibliographystyle{abbrv}
\bibliography{nsc_cv_mac}

\end{document}